\documentclass[11pt,a4paper]{amsart}
\usepackage{amssymb,amsmath,epsfig,graphics,mathrsfs}

\usepackage{fancyhdr}
\usepackage{hyperref}
\hypersetup{
 colorlinks   = true,
 urlcolor     = blue,
 linkcolor    = blue,
 citecolor   = red ,
 bookmarksopen=true
}

\usepackage{amsmath}
\usepackage{amsfonts}
\usepackage{amssymb}
\usepackage{amsthm}
\usepackage{epsfig,graphics,mathrsfs}
\usepackage{graphicx}

\usepackage[usenames, dvipsnames]{color} 

\usepackage{hyperref}

\def \de {\partial}
\def \e {\varepsilon}
\def \N {\mathbb{N}}

\def \phi {\varphi}
\def \RNu {\mathbb{R}^{N+1}}
\def \RN {\mathbb{R}^N}
\def \R {\mathbb{R}}

\def \K {\mathscr{K}}

\def \G{\Gamma}

\def \vf{\varphi}

\def \So {\mathscr{S}}
\newcommand{\As}{(-\mathscr A)^s}
\newcommand{\sA}{\mathscr A}

\newcommand{\Bpa}{B^{\alpha,p}\left(\RN\right)}
\newcommand{\Bs}{B^{2s,1}\left(\RN\right)}
\newcommand{\Ia}{\mathscr I_\alpha}

\newcommand{\rpp}{\rho_p(\sA)}

\newcommand{\Rn}{\mathbb R^n}

\newcommand{\p}{\partial}

\newcommand{\la}{\lambda}

\numberwithin{equation}{section}

\newcommand{\beq}{\begin{equation}}
\newcommand{\bea}[1]{\begin{array}{#1} }
\newcommand{\eeq}{ \end{equation}}
\newcommand{\ea}{ \end{array}}

\newcommand{\ve}{\varepsilon}

\newcommand{\Lo}{\mathscr L^{2s,p}}

\newcommand{\Ps}{\mathfrak P_s^{\sA}}
\newcommand{\In}{\mathbf 1_E}
\newcommand{\Lp}{L^p}

\newcommand{\Lii}{L^\infty_0}

\newtheorem{theorem}{Theorem}[section]
\newtheorem{lemma}[theorem]{Lemma}
\newtheorem{proposition}[theorem]{Proposition}
\newtheorem{corollary}[theorem]{Corollary}
\newtheorem{remark}[theorem]{Remark}
\newtheorem{definition}[theorem]{Definition}

\numberwithin{equation}{section}

\begin{document}

\title[Nonlocal isoperimetric inequalities, etc.]{Nonlocal isoperimetric inequalities for Kolmogorov-Fokker-Planck operators}

\date{}

\begin{abstract}

In this paper we establish optimal isoperimetric inequalities for a nonlocal perimeter adapted to the fractional powers of a class of  Kolmogorov-Fokker-Planck operators which are of interest in physics. These operators are very degenerate and do not possess a variational structure. The prototypical example was introduced by Kolmogorov in his 1938 paper on Brownian motion and the theory of gases. Our work has been influenced by ideas of M. Ledoux in the local case.   
\end{abstract}

\author{Nicola Garofalo}

\address{Dipartimento d'Ingegneria Civile e Ambientale (DICEA)\\ Universit\`a di Padova\\ Via Marzolo, 9 - 35131 Padova,  Italy}
\vskip 0.2in
\email{nicola.garofalo@unipd.it}

\thanks{The first author was supported in part by a Progetto SID (Investimento Strategico di Dipartimento) ``Non-local operators in geometry and in free boundary problems, and their connection with the applied sciences", University of Padova, 2017.}

\author{Giulio Tralli}
\address{Dipartimento d'Ingegneria Civile e Ambientale (DICEA)\\ Universit\`a di Padova\\ Via Marzolo, 9 - 35131 Padova,  Italy}
\vskip 0.2in
\email{giulio.tralli@unipd.it}

\maketitle

\tableofcontents

\section{Introduction}\label{S:intro}

Isoperimetric inequalities is a subject with roots in the classical antiquity, but which presently continues to be an active source of inspiration in analysis and geometry. The classical isoperimetric inequality, the so-called problem
of Queen Dido in Virgil's \emph{Aeneid}, states that among all measurable sets $E\subset \RN$ with given perimeter the ball is the one with largest volume. More precisely, denoting by $\omega_N$ the $N$-dimensional Lebesgue measure of the unit ball, one has
\[
(\star)\ \ \ \ \ \ P(E) \ge N \omega_N^{1/N} |E|^{\frac{N-1}N},
\]
and equality holds if and only if $E$ is a ball. Here, the notion of perimeter is the one introduced by De Giorgi in \cite{DG54}, who also first provided in \cite{DG} a complete proof of $(\star)$ with the case of equality. 
His original formulation was based on the regularising properties of the heat semigroup $P_t = e^{-t \Delta}$. On a measurable set $E\subset \RN$ with finite measure, he defined his perimeter as  $P(E) = \underset{t\to 0^+}{\lim}\ ||\nabla P_t \mathbf 1_E||_1$,
where $\mathbf 1_E$ denotes the indicator function of $E$. This notion coincides with the well-known variational formulation 
$P(E) = \underset{\vf\in \Phi}{\sup} \int_{\RN} \mathbf 1_E \operatorname{div} \vf$, where $\Phi = \{\vf\in C^1_0(\RN,\RN)\mid ||\vf||_\infty\le 1\}$,
see e.g. \cite{Gi}. Sets for which $P(E)<\infty$ are called Caccioppoli sets, and the study of the structure of sets which minimise the perimeter, the minimal surfaces, has been one the main engines behind the development of geometric measure theory.     

In recent years, there has been considerable interest in geometric objects that can be interpreted as a non-infinitesimal version of classical minimal surfaces. For instance, they arise in the study of surface tension in two-phase systems such as snowflakes or dendritic formations, see \cite{Vi}. In their seminal work \cite{CRS} Caffarelli, Roquejoffre and Savin introduced the concept of a nonlocal minimal surface and studied the structure of such sets. Their starting point is the following notion: given a number $0<s<1/2$, a measurable set $E\subset \RN$ is said to have finite $s$-perimeter, if
\begin{align}\label{crs0}
P_s(E)  \overset{def}{=} &  \int_{\RN}\int_{\RN} \frac{|\mathbf 1_E(X) - \mathbf 1_E(Y)|^2}{|X-Y|^{N+2s}} dX dY  = 2 \int_{\RN\setminus E}\int_{E} \frac{dX dY}{|X-Y|^{N+2s}} < \infty.
\end{align}
(one should keep in mind that smooth bounded sets have finite $s$-perimeter when $s < 1/2$, whereas no open set $E\not=\varnothing$ has finite $s$-perimeter when $s = 1/2$, see  \cite[Lemma 3.2]{Sickel}). Their main result, see \cite[Theorem 2.4]{CRS}, shows that a nonlocal minimal surface is $C^{1,\alpha}$ in the neighbourhood of any of its points, except for a $(N-2)-$dimensional closed set. Throughout this paper we assume $N\ge 2$.

It is worth noting here that, according to \eqref{crs0}, a measurable set $E\subset \RN$ has finite $s$-perimeter if and only if $\mathbf 1_E\in W^{s,2}$, where for $1\le p < \infty$ and $s>0$ we have denoted by $W^{s,p}$ the Banach space of functions $f\in \Lp$ with finite Aronszajn-Gagliardo-Slobedetzky seminorm, 
\begin{equation}\label{ags}
[f]_{p,s} = \left(\int_{\RN} \int_{\RN} \frac{|f(X) - f(Y)|^p}{|X-Y|^{N+ps}} dX dY\right)^{1/p}.
\end{equation}
In fact, it is clear from \eqref{crs0} and \eqref{ags} that 
\begin{equation}\label{psE}
P_s(E) = [\mathbf 1_E]^2_{2,s},
\end{equation}
and this also underscores the variational nature of the connection between the fractional perimeter and the nonlocal Laplacian 
\[
(-\Delta)^s f(X) = \gamma(N,s) \int_{\RN} \frac{f(X) - f(Y)}{|X-Y|^{N+2s}} dY
\]
(here $\gamma(N,s) >0$ is a suitable normalising constant, and the integral must be interpreted in the principal value sense, see e.g. \cite[Prop. 5.6]{Gft}). In fact, it is easy to verify that $(-\Delta)^s f = 0$ is the Euler-Lagrange equation of the energy functional $\mathscr E_s(f) = \frac{\gamma(N,s)}2  [f]^2_{2,s}.$ If we denote by $\delta_\la(X) = \la X$ the Euclidean dilations, from \eqref{psE}, and the scaling property $[\delta_\la f]^p_{p,s} = \la^{-N+ps}[f]^p_{p,s}$
of \eqref{ags},  we infer the scale invariance  of the quotient
\[
\frac{P_s(\delta_\la E)}{|\delta_\la E|^{\frac{N-2s}N}} =  \frac{P_s(E)}{|E|^{\frac{N-2s}N}}.
\]
This leads to conjecture the following nonlocal version of the isoperimetric inequality $(\star)$: given $0<s<1/2$, there exists a constant $i(N,s)>0$ such that for any bounded measurable set $E\subset \RN$,  one has
\begin{equation}\label{isos}
P_s(E)\ \ge\ i(N,s)\ |E|^{(N-2s)/N}.
\end{equation}
The inequality \eqref{isos} is in fact true, and, interestingly, can be obtained by a 1989 result of Almgren and Lieb. In their \cite[Theorem 9.2 (i)]{AL} these authors proved that, if for $0<s<1$ and $1\le p<\infty$ one has $f\in W^{s,p}$, then also $f^\star\in W^{s,p}$ and 
\begin{equation}\label{AL}
||f^\star||_{W^{s,p}} \le ||f||_{W^{s,p}}.
\end{equation}
Here, $f^\star$ denotes the non-increasing rearrangement of $|f|$. If we apply \eqref{AL} to $f = \mathbf 1_E$ with $p=2$, and keep in mind \eqref{psE}, and the fact that $\mathbf 1^\star_{E} = \mathbf 1_{E^\star}$ (where $E^\star$ denotes the ball in $\RN$ centred at the origin with measure $|E|$), we obtain
\[
P_s(E) \ge  [\mathbf 1^\star_{E}]^2_{2,s} = [\mathbf 1_{E^\star}]^2_{2,s} = P_s(E^\star).
\]
The right-hand side of the latter inequality is finite if and only if $0<s<1/2$, and can be shown to equal an explicitly computable multiple of $|E|^{(N-2s)/N}$, see \cite{GAL}.

In this paper we establish nonlocal isoperimetric inequalities such as \eqref{isos}, when $\Delta$ is replaced by a class of operators that do not possess a variational structure (a gradient) or a homogeneous structure (dilations), and which can in general be very degenerate. Because of their interest in mathematics and physics these operators have been intensely studied during the past three decades, but a theory of isoperimetric inequalities has so far been lacking and there presently exist no results akin to those in this paper. Specifically, we consider the following class of differential operators in $\RN$, 
\begin{equation}\label{A0}
\mathscr A u  \overset{def}{=} \operatorname{tr}(Q \nabla^2 u) + <BX,\nabla u>,
\end{equation}
where the $N\times N$ matrices $Q$ and $B$ have real, constant coefficients, $Q = Q^\star \ge 0$. In the non-degenerate case when $Q = I_N, B = O_N$, then $\sA = \Delta$, and we are back into the framework of \cite{CRS}, \cite{AL}. Other than for illustrative purposes, we will not be interested in this case. Our focus instead is when $Q$ is not invertible and $B\not= O_N$. In such case the operator $\sA$ is degenerate and, because of the drift, it does not possess a variational or homogeneous structure. We recall that in the opening of his celebrated work \cite{Ho} H\"ormander proved that \eqref{A0} is hypoelliptic if and only if its covariance matrix satisfies the following hypothesis for every $t>0$
\begin{equation}\label{Kt}
K(t) = \frac 1t \int_0^t e^{sB} Q e^{s B^\star} ds\ >\ 0.
\end{equation} 
Under such condition he proved that for every $f\in \So$ the Cauchy problem $\sA u - \p_t u = 0$ in $\RN\times (0,\infty)$, $u(X,0) = f(X)$, admits a unique solution $u(X,t) = P_t f(X) = \int_{\RN} p(X,Y,t) f(Y) dY$. Furthermore, $\sA - \p_t$ possesses the following strictly positive explicit fundamental solution
\begin{equation}\label{p}
p(X,Y,t) = \frac{c_N}{V(t)} \exp\left( - \frac{m_t(X,Y)^2}{4t}\right),
\end{equation} 
where $m_t(X,Y)$ is the time-dependent intertwined pseudo-distance 
\[
m_t(X,Y)  = \sqrt{<K(t)^{-1}(Y-e^{tB} X ),Y-e^{tB} X >},\ \ \ \ \ \ \ t>0,
\] 
and
\[
V(t) = \operatorname{Vol}_N(B_t(X,\sqrt t)) = \omega_N  (\det(t K(t)))^{1/2},
\]
where
$$B_t(X,r) = \{Y\in \RN\mid m_t(X,Y) < r\},\ \ \ \ \ \ \ r>0.$$
For any $1\le p<\infty$, the family $\{P_t\}_{t>0}$ is a strongly continuous semigroup on $\Lp$ which is non-symmetric (unless $B = O_N$) and in general non-doubling, see \cite{GThls}.

In the present paper we focus on the subclass of operators in \eqref{A0} which, besides the hypoellipticity condition \eqref{Kt}, also satisfy the assumption
\begin{equation}\label{trace}
\operatorname{tr} B \ge 0.
\end{equation} 
This hypothesis guarantees that the semigroup be contractive in $\Lp$ for $1\le p<\infty$. This aspect played a key role in our work \cite{GT}, in which we developed a calculus of the nonlocal operators $\As$. Such calculus, which has been instrumental to our recent work \cite{GThls} on Hardy-Littlewood-Sobolev inequalities, is central to the present paper as well. But before we can introduce our main results, we need to clarify the role of \eqref{trace} in connection with another important aspect of the analysis of the semigroup $P_t = e^{-t\sA}$: the large-time behaviour of the volume function $V(t)$ in \eqref{p}. In this respect, we recall that in \cite{GThls} we have introduced the notion of intrinsic dimension at infinity as the extended number $D_\infty = \sup\ \Sigma_\infty$, where 
\[
\Sigma_\infty = \bigg\{\alpha>0\big| \int_1^\infty \frac{t^{\alpha/2-1}}{V(t)} dt < \infty\bigg\}.
\]
When $\Sigma_\infty = \varnothing$ we set $D_\infty = 0$, otherwise we obviously have $0<D_\infty \le \infty$. We note that if $V(t) \cong t^{D/2}$ as $t\to \infty$, then $(0,D) =  \Sigma_\infty$, and therefore $D_\infty =D$. We also stress that, in the absence of \eqref{trace}, it can happen that $D_\infty = 0$. In fact, if we take $Q = I_N, B = - I_N$ in \eqref{A0}, then $\sA$ is the classical Ornstein-Uhlenbeck operator, for which a computation shows that $V(t) = c_N (1-e^{-2t})^{\frac{N}{2}} \to c_N$, for some $c_N>0$. Therefore, in such case $\Sigma_\infty = \varnothing$. However, \eqref{trace} does not hold for this example, which therefore remains outside the scope of the present work. Concerning $D_\infty$, we mention that it was proved in \cite[Prop. 3.1]{GThls} that if \eqref{trace} holds, the following is true:
\begin{itemize}
\item[(i)] there exists a constant $c_1>0$ such that $V(t)\geq c_1 t$ for all $t\geq 1$;
\item[(ii)] moreover, if $\max\{\Re(\lambda)\mid \lambda\in \sigma(B)\}=L_0>0$, then there exists a constant $c_0$ such that $V(t)\geq c_0 e^{L_0 t}$ for all $t\geq 1.$
\end{itemize}
An immediate consequence of (i) and of the definition of $\Sigma_\infty$, is that $D_\infty \ge 2$ is always true under the hypothesis \eqref{trace}. Furthermore, if (ii) occurs, then $D_\infty = \infty$. We stress that such case can occur, see the Ex. $6^+$ in the table in fig. 1 in \cite{GThls}. The blowup of the volume function $V(t)$ as $t\to \infty$ plays a pervasive role in the analysis of \eqref{A0} when combined with the following $L^p-L^\infty$ ultracontractivity of the semigroup: for any $1\le p\le \infty$ one has for $f\in L^p(\RN)$,
\begin{equation}\label{uc}
|P_t f(X)| \le \frac{c_{N,p}}{V(t)^{1/p}} ||f||_{p},
\end{equation}
for a certain constant $c_{N,p}>0$, see \cite[Section 3]{GThls}. We note that, in view of (i), it follows from \eqref{uc} that, when \eqref{trace} holds, we must have $P_t f(X) \to 0$ as $t\to \infty$, for every $X\in \RN$.

With all this being said, we now turn to the description of the results in this paper. As a first step we introduce for the operators $\sA$ a generalisation of the notion of $s$-perimeter \eqref{crs0}. Since these operators lack a variational structure, we circumvent this difficulty using a relaxation procedure of the functional $f \to ||\As f||_1$. Precisely, given a function $f\in L^1$ we denote by $\mathscr F(f)$ the collection of all sequences $\{f_k\}_{k\in \mathbb N}$ in $\So$ such that $f_k \to f$ in $L^1$. If $E\subset \RN$ is a measurable set such that $|E|<\infty$, we define the $s$-\emph{perimeter} of $E$ by the formula
\begin{equation}\label{sperimeter0}
\Ps(E) \overset{def}{=}  \underset{\{f_k\}_{k\in \mathbb N}\in \mathscr F(\mathbf 1_E)}{\inf}\ \underset{k\to \infty}{\liminf}\ ||\As f_k||_1.
\end{equation}
If there exists at least one sequence $\{f_k\}_{k\in \mathbb N}\in \mathscr F(\In)$ such that $\underset{k\to \infty}{\liminf}\ ||\As f_k||_1<\infty$, then we clearly have $\Ps(E)<\infty$ and we say that $E$ has finite $s$-perimeter. With such notion in hands, the first question that comes to mind is the connection between \eqref{sperimeter0} and that of Caffarelli, Roquejoffre and Savin in \eqref{psE}. In Proposition \ref{P:equalstars} we show that, in the non-degenerate case when $Q = I_N$ and $B = O_N$, and thus $\sA = \Delta$ in \eqref{A0}, for every $0<s<1/2$ our $s$-perimeter coincides (up to an explicit universal constant $c(N,s)>0$) with that in  \cite{CRS},  i.e., we have
\begin{equation}\label{sameper}
\mathfrak P_s^\Delta(E) = c(N,s) P_s(E).
\end{equation}

Having clarified this point, we are ready to discuss the main results of this paper.

\begin{theorem}[First nonlocal isoperimetric inequality]\label{T:iso01}
Suppose that the hypothesis \eqref{trace} hold. If there exist $D, \gamma_D>0$ such that for every $t>0$ one has
\begin{equation}\label{vol0}
V(t) \ge \gamma_D\ t^{D/2},
\end{equation}
then for any $0<s< 1/2$ there exists a constant $i(s) >0$, depending on $N,D,s,\gamma_D$, such that for any measurable set $E\subset \RN$, with $|E|<\infty$, one has
\begin{equation}\label{isoper20}
\Ps(E)\ \ge\ i(s)\ |E|^{(D-2s)/D}.
\end{equation}
\end{theorem}

Before proceeding with our second main result, we pause to note that Theorem \ref{T:iso01} encompasses the Almgren-Lieb's isoperimetric inequality \eqref{isos} as a special case, and provides an alternative semigroup-based proof of the latter. To see this, we take $Q = I_N, B = O_N$ in \eqref{A0}, so that $\sA = \Delta$, and $p(X,Y,t) = (4\pi t)^{-N/2} \exp(-\frac{|X-Y|^2}{4t})$ is the standard heat kernel.
Since in such case we trivially have $V(t) = c_N t^{N/2}$, it is obvious that \eqref{vol0} holds with equality if $D = N$. We thus obtain from \eqref{sameper} and \eqref{isoper20},
\begin{equation}\label{isoper30}
\mathfrak P_s^\Delta(E)\ \ge\ \alpha(N,s)\ |E|^{(N-2s)/N},
\end{equation}
for some universal constant $\alpha(N,s)>0$, which is \eqref{isos}.

A more important prototypical example to keep in mind is the diffusion operator in $\RN$
\[
\sA_0 u = \Delta_v u + <v,\nabla_x u>,
\]
whose evolutive counterpart is $\K_0 = \Delta_v u + <v,\nabla_x u> - \p_t u$. Here, we have let $N = 2n$, and $X = (v,x)$, with $v, x\in \Rn$. We observe that, in this example, $\operatorname{tr} B = 0$ (and thus \eqref{trace} is trivially verified), and that $\sA_0$ is invariant under the non-isotropic dilations $(v,x) \to (\la v, \la^3 x)$. The operator $\K_0$ was introduced by Kolmogorov in his seminal 1934 note \cite{Kol} on Brownian motion and the theory of gases. This operator badly fails to be parabolic since it is missing the diffusive term $\Delta_x u$, but it does satisfy H\"ormander's hypoellipticity condition \eqref{Kt}. In fact, one easily checks that $K(t)=\begin{pmatrix} I_n & t/2\ I_n\\ t/2\ I_n& t^2/3\ I_n\end{pmatrix}>0$ for every $t>0$. Kolmogorov himself found the following explicit fundamental solution
\begin{align*}
p_0(X,Y,t) & = \frac{c_n}{t^{2n}} \exp\big\{- \frac 1t \big(|v-w|^2 
 + \frac 3t <v-w,y-x-tv>  + \frac{3}{t^2} |x- y +tv|^2\big)\big\}.
\end{align*}
Comparing with \eqref{p}, we see that $V(t) = \alpha_n t^{2n}$, and thus \eqref{vol0} holds with equality, with $D = 4n$. We conclude from \eqref{isoper20} that for every $0<s<1/2$ there exists a constant $i(n,s)>0$, such that for any measurable set $E\subset \R^{2n}$, with $|E|<\infty$, one has
\begin{equation}\label{isokolmo0}
\mathfrak P^{\sA_0}_s(E)\ \ge\ i(n,s)\ |E|^{(2n-s)/2n}.
\end{equation}
We emphasise that, because of the non-isotropic dilations $(v,x) \to (\la v, \la^3 x)$, the exponent $\frac{2n-s}{2n}$ in the right-hand side of \eqref{isokolmo0} is best possible. Because of their interest in mathematics and physics, Kolmogorov's operator and its variants have been intensely studied over the past three decades (we refer to \cite{GThls} for a bibliographical account), but isoperimetric inequalities such as \eqref{isokolmo0} have been completely missing. 
 
Hypoelliptic operators modeled on Kolmogorov's provide local approximating homogeneous structures for the general class \eqref{A0}. This fact, discovered in \cite{LP}, allows to infer the existence of a number $D_0\ge N\ge 2$ such that $V(t) \cong t^{D_0/2}$ as $t\to 0^+$, where $V(t)$ is as in \eqref{VS}. We call such number the intrinsic dimension at zero of the semigroup. 
To prepare the discussion of our second main result we emphasise at this point that the assumption \eqref{vol0} in Theorem \ref{T:iso01} imposes the restriction
\[
D_0 \le D \le D_\infty.
\]
It ensues that such result does not cover the situation in which $D_0 > D_\infty$. We stress that, although this case never happens when \eqref{A0} possesses an underlying homogeneous structure, in general such phenomenon can occur. Consider for instance the degenerate operator in $\R^2$
\[
\sA_1 u = \p_{vv} u - x \p_v u + v \p_x u,
\]
whose evolutive counterpart is the Kramers' operator $\K_1 u = \p_{vv} u - x \p_v u + v \p_x u - \p_t u$. In this example we have $n = 1$, $N = 2n = 2$, $Q = \begin{pmatrix} 1 & 0 \\ 0 & 0\end{pmatrix}$, $B = \begin{pmatrix} 0 & -1 \\ 1 & 0\end{pmatrix}$. A computation gives $V(t) = \pi\left(\frac{t^2}{4}+\frac{1}{8}\left(\cos(2t)-1\right)\right)^{\frac{1}{2}}$, from which we see that $D_\infty = 2$. On the other hand, the intrinsic dimension at $t = 0$ is the same as that of the (homogeneous) Kolmogorov operator $\K_0 u= \p_{vv} u + v \p_x u - \p_t u$ in $\R^3$, and therefore $D_0 = 4 > D_\infty$. 
This leads us to introduce the second main result in this paper.

\begin{theorem}[Second nonlocal isoperimetric inequality]\label{T:isoother0}
Suppose that $D_0 > D_\infty$, and that for some $\gamma>0$ we have for every $t>0$
\begin{equation}\label{vol20}
V(t) \ge \gamma \min\{ t^{D_0/2},t^{D_\infty/2} \}.
\end{equation} 
Given $0<s< 1/2$, there exists a constant $i(s) >0$, depending on $N,D_0, D_\infty, s,\gamma$, such that for any measurable set $E\subset \RN$, with $|E|<\infty$, one has
\begin{equation}\label{isoperother20}
\Ps(E) \geq\ i(s) \min\left\{|E|^{\frac{D_0-2s}{D_0}}, |E|^{\frac{D_\infty-2s}{D_\infty}}\right\}.
\end{equation}
\end{theorem}

Our third main result represents a notable application of the nonlocal isoperimetric inequality in Theorem \ref{T:iso01}. To provide the reader with the proper historical perspective, we recall that, for the classical Sobolev spaces,  from the representation formula $|f(X)|\le C(N) \int_{\RN} \frac{|\nabla f(Y)|}{|X-Y|^{N-1}} dY$, and the $L^1$ mapping properties of the Riesz potentials, one knows that $W^{1,1}(\RN) \hookrightarrow L^{N/(N-1),\infty}$. A remarkable aspect of the end-point case $p=1$ is that such weak Sobolev embedding in fact implies the isoperimetric inequality $P(E) \ge C_N |E|^{\frac{N-1}N}$. The latter, in turn, combined with the coarea formula, is equivalent to the strong embedding $W^{1,1}(\RN) \hookrightarrow L^{N/(N-1)}$. This circle of ideas establishes the beautiful fact that, in the geometric case $p=1$, the weak Sobolev embedding is equivalent to the strong one, and they are both equivalent to the isoperimetric inequality, see \cite{Ma2}. We establish a semigroup generalisation of this fact to the nonlocal degenerate setting of this paper. We in fact obtain the following endpoint result for an optimal class of Besov spaces, introduced in \cite[Section 3]{GTfi}, and naturally associated with the operator $\sA$ in \eqref{A0}, see Definition \ref{D:besov} below. 

\begin{theorem}\label{T:strongsob0}
Let $s\in(0,\frac{1}{2})$. Suppose that \eqref{trace} be valid, and that there exist $D, \gamma_D>0$ such that \eqref{vol0} hold. Then, we have
$$\Bs\hookrightarrow L^{\frac{D}{D-2s}}\left(\RN\right).$$ 
Precisely, for every $f\in \Bs$ one has
\begin{equation}\label{ssob}
||f||_{\frac{D}{D-2s}} \le \frac{s}{i(s)\G(1-s)} \mathscr N_{2s,1}(f),
\end{equation}
where $i(s)>0$ is the constant appearing in Theorem \ref{T:iso01}, and $\mathscr N_{2s,1}(f)$ denotes the Besov seminorm in \eqref{besov} below.
\end{theorem}
\noindent We refer the reader to Section \ref{S:app} for the counterpart of Theorem \ref{T:strongsob0} under the assumption \eqref{vol20}.

A description of the organisation of the paper is in order at this point. 
In Section \ref{S:prelim} we collect some well-known preliminary background that is used in the rest of the paper.  In Section \ref{S:ind} we analyse the action of the H\"ormander semigroup $\{P_t\}_{t>0}$ on indicator functions. The main result is Lemma \ref{L:perbelow}, which generalises a result originally due to Ledoux \cite{Led}. Section \ref{S:per} introduces the notion of nonlocal variation of a function in $L^1$ and of fractional perimeter, see Definition \ref{D:BV}. In Definition \ref{per2} we define a second notion of fractional perimeter, $\mathfrak P^{\sA,\star}_s(E)$, inspired to that originally given by De Giorgi in the local case. We show in Lemma \ref{L:senzastarminoredistar}
that
\[
\mathfrak P^\sA_s(E) \le \mathfrak P^{\sA,\star}_s(E),
\] 
but we do not presently know whether these two nonlocal perimeters coincide in general. However, Proposition \ref{P:equalstars} shows that, in the non-degenerate case when $Q = I_N$ and $B = O_N$, and thus $\sA = \Delta$ in \eqref{A0}, for every $0<s<1/2$ we have
$$\mathfrak P_s^\Delta(E) = \mathfrak P_s^{\Delta,\star}(E),$$
see also \eqref{sameper}. Section \ref{S:ledoux} contains a key estimate inspired to one originally obtained by M. Ledoux, see \eqref{led}, or also (8.14) in his contribution in \cite{DGL}. Theorem \ref{T:generaled} states that for every $f\in \So$ and $t, \tau>0$ one has for any $0<s<1$ and $1\le p<\infty$,
\[
||P_t f - P_\tau f||_{p} \le \frac{2 |t-\tau|^s}{\G(1+s)}\ \underset{\sigma>0}{\sup}\ ||\As P_\sigma f||_p.
\]
The case $p=1$ of this result plays a critical role in the proofs of our main results. In Section \ref{S:main} we prove Theorems \ref{T:iso01} and \ref{T:isoother0}. Finally, in Section \ref{S:app} we establish Theorem \ref{T:strongsob0} and Theorem \ref{T:strongsob2}.
 
In closing we mention that, as a special case of their celebrated works, in \cite[Cor. 5]{BBM1}, see also \cite{BBM2}, \cite{Bc} and \cite{MS},  Bourgain, Brezis and Mironescu obtained a new characterisation of BV, and therefore of De Giorgi's perimeter, based on their two sided bound
\begin{equation}\label{anotherlook}
C_1 P(E) \le \underset{s\nearrow 1/2}{\liminf}\ (1/2 -s) P_s(E) \le \underset{s\nearrow 1/2}{\limsup}\ (1/2 -s) P_s(E) \le  C_2 P(E).
\end{equation}
Answering a question posed in \cite{BBM1}, D\'avila in \cite[Theor.1]{Davila} refined the limiting formula \eqref{anotherlook}, and  proved that
\begin{equation}\label{dav}
\underset{s\nearrow 1/2}{\lim}\ (1/2 -s) P_s(E) = \left(\frac 12 \int_{\mathbb S^{N-1}} |<e_N,\omega>|\right)\ P(E) = \frac{\pi^{\frac{N-1}2}}{\G(\frac{N+1}2)}\ P(E),
\end{equation}
where $e_N = (0,...,0,1)$. The upper bound in \eqref{anotherlook} can also be extracted from a subsequent inequality of Maz'ya in \cite{Ma1}. The case of equality in \eqref{isos} was obtained in \cite{FS} as a consequence of their general results on Hardy inequalities. The limiting behaviour of the fractional perimeter was also studied in \cite{ADM} and \cite{CV}. In connection with \eqref{isos} above, and still in a non-degenerate context, more general nonlocal isoperimetric inequalities have been considered in the works \cite{FMM}, \cite{Lud}, \cite{FFMMM} and \cite{CN}. 

We mention that in Proposition \ref{P:star} we obtain an upper bound, such as that in \eqref{anotherlook}, for our perimeter $\mathfrak P^{\sA,\star}_s$, and consequently also for $\Ps$. We plan to address the precise limiting behaviour of these nonlocal perimeters in a future study. 
Finally, it would be of considerable interest to understand the structure of nonlocal minimal surfaces for the class \eqref{A0}. 

\medskip

\noindent \emph{Acknowledgment:} We thank Giorgio Metafune for his kind help with the proof of Proposition \ref{density}.

\section{Preliminaries}\label{S:prelim}

In this section we introduce the relevant notation and recall some well-known material concerning the class \eqref{A0} that will be used in the rest of the paper. For details we refer the reader to \cite[Sec. 2]{GT} and \cite[Sec. 2]{GThls}.

\subsection{Notation} 

Generic points in $\RN$ will be denoted with the letters $X, Y, Z$. Points in $\RNu$, with $(X,t), (Y,t)$. The trace of a matrix $A$ will be indicated with $\operatorname{tr} A$, $A^\star$ is the transpose of $A$. The Hessian matrix of a function $u$ is denoted by $\nabla^2 u$.  Given a set $E\subset \RN$, we denote with $\In$ its indicator function. If $E$ is measurable, we denote by $|E|$ its $N$-dimensional Lebesgue measure. All the function spaces in this paper are based on $\RN$, thus we will routinely avoid reference to the ambient space throughout this work. For instance, the Schwartz space of rapidly decreasing functions in $\RN$ will be denoted by $\So$, and for $1\le p \le \infty$ we let $\Lp = L^p(\RN)$. The norm in $\Lp$ will be denoted by $||\cdot||_p$, instead of $||\cdot||_{\Lp}$. We will indicate with $\Lii$ the Banach space of the $f\in C(\RN)$ such that $\underset{|X|\to \infty}{\lim}\ |f(X)| = 0$ with the norm $||\cdot||_\infty$. The reader should keep in mind the following simple facts: (1) $P_t : L^\infty_0 \to L^\infty_0$ for every $t>0$; (2) $\So$ is dense in $\Lii$.
  If $T:\Lp\to L^q$ is a bounded linear map, we will indicate with $||T||_{p\to q}$ its operator norm. If $ q =p$, the spectrum of $T$ on $\Lp$ will be denoted by $\sigma_p(T)$, the resolvent set by $\rho_p(T)$, the resolvent operator by $R(\la,T) = (\la I - T)^{-1}$. For $x>0$ we will indicate with $\G(x) = \int_0^\infty t^{x-1} e^{-t} dt$ Euler's gamma function.
For any $N\in \mathbb N$ we will use the standard notation
$\sigma_{N-1} = \frac{2\pi^{N/2}}{\G(N/2)}$, $\omega_N = \frac{\sigma_{N-1}}{N}$,
respectively for the $(N-1)$-dimensional measure of the unit sphere $\mathbb S^{N-1} \subset \RN$, and $N$-dimensional measure of the unit ball. We adopt the convention that $a/\infty = 0$ for any $a\in \R$.

\subsection{Semigroup matters}\label{SS:m}

Given matrices $Q$ and $B$ as in \eqref{A0} we consider the semigroup $P_t f(X) = e^{-t \sA} f(X) = \int_{\RN} p(X,Y,t) f(Y) dY$, where $p(X,Y,t)$ is given by \eqref{p}. As in \cite{GThls}, for $X, Y\in \RN$ we have defined
\begin{align}\label{m}
m_t(X,Y) & = \sqrt{<K(t)^{-1}(Y-e^{tB} X ),Y-e^{tB} X >},\ \ \ \ \ \ \ t>0.
\end{align}
It is obvious that, when $B\not= O_N$, we have $m_t(X,Y) \not= m_t(Y,X)$ for every $t>0$. Given $X\in \RN$ and $r>0$, we consider the set 
$$B_t(X,r) = \{Y\in \RN\mid m_t(X,Y) < r\},$$
and call it the time-varying pseudo-ball. One has 
\begin{equation}\label{misB}
\operatorname{Vol}_N(B_t(X,r)) = \omega_N r^N (\det K(t))^{1/2}. \end{equation}
We stress that the quantity in the right-hand side of \eqref{misB} is independent of $X\in \RN$, a reflection of the underlying group structure $(X,s)\circ (Y,t) = (Y+ e^{-tB}X,s+t)$ induced by the drift matrix $B$ (see \cite{LP}). Endowed with the latter, the space $(\R^{N+1},\circ)$ becomes a non-Abelian Lie group. This aspect is reflected in the expression \eqref{p}, as well as in  \eqref{misB}.
As a consequence, we can drop the dependence in $X$, and indicate 
$\operatorname{Vol}_N(B_t(X,r)) = V_t(r)$.
When $r = \sqrt t$, we simply write $V(t)$, instead of $V_t(\sqrt t)$, i.e., 
\begin{equation}\label{VS}
V(t) = \operatorname{Vol}_N(B_t(X,\sqrt t)) = \omega_N  (\det(t K(t)))^{1/2}.
\end{equation}

In the following lemmas we collect the main (well-known) properties of the semigroup $\{P_t\}_{t>0}$. 

\begin{lemma}\label{L:invS}
For any $t>0$ we have: 
\begin{itemize}
\item[(a)] $\sA(\So)\subset \So$ and $P_t(\So) \subset \So$;
\item[(b)] For any $f\in \So$ and $X\in \RN$ one has $\frac{\de}{\de t} P_t f(X) = \mathscr A P_t f(X)$; 
\item[(c)] For every $f\in \So$ and $X\in \RN$ the commutation property is true
$\mathscr A P_t f(X) = P_t \mathscr A  f(X)$.
\end{itemize}
\end{lemma}

\begin{lemma}\label{L:Pt}
The following properties hold:
\begin{itemize}
\item[(i)] For every $X\in \RN$ and $t>0$ we have
$P_t 1(X) = \int_{\RN} p(X,Y,t) dY = 1$;
\item[(ii)] $P_t:L^\infty \to L^\infty$ with $||P_t||_{L^\infty\to L^\infty} \le 1$;
\item[(iii)] For every $Y\in \RN$ and $t>0$ one has
$\int_{\RN} p(X,Y,t) dX = e^{- t \operatorname{tr} B}$.
\item[(iv)] Let $1\le p<\infty$, then $P_t:L^p \to L^p$ with $||P_t||_{L^p\to L^p} \le e^{-\frac{t \operatorname{tr} B}p}$. If $\operatorname{tr} B\ge 0$, $P_t$ is a contraction on $L^p$ for every $t>0$;
\item[(v)] [Chapman-Kolmogorov equation]
For every $X, Y\in \R^N$ and $t>0$ one has
$p(X,Y,s+t)  = \int_{\R^N} p(X,Z,s) p(Z,Y,t) dZ$.
Equivalently, one has $P_{t+s} = P_t \circ P_s$ for every $s, t>0$.
\end{itemize}
\end{lemma}

\begin{lemma}\label{L:Lprate}
Let $1\le p \le \infty$. Given any $f\in \So$ for any $t\in [0,1]$ we have 
\[
||P_t f - f||_{p} \le ||\mathscr A f||_{p}\ \max\{1,e^{-\frac{\operatorname{tr} B}p}\}\ t.
\]
\end{lemma}

\begin{corollary}\label{C:Ptpzero}
Let $1\le p< \infty$. For every $f\in L^p$, we have
$||P_tf-f||_{p}\rightarrow 0$ as $t \to 0^+.$ Consequently, $\{P_t\}_{t>0}$ is a strongly continuous semigroup on $\Lp$. The same is true when $p = \infty$, if we replace $L^\infty$ by the space $\Lii$.  
\end{corollary}

\begin{remark}\label{R:infty}
The reader should keep in mind that from this point on when we consider $\{P_t\}_{t>0}$ as a strongly continuous semigroup in $\Lp$, we always intend to use $\Lii$ when $p = \infty$.
\end{remark}

Denote by $(\sA_p,D_p)$ the infinitesimal generator of the semigroup $\{P_t\}_{t>0}$ on $L^p$ with domain  
\begin{equation}\label{Dp}
D_p = \big\{f\in L^p\mid \sA_p f \overset{def}{=} \underset{t\to 0^+}{\lim}\ \frac{P_t f - f}{t}\ \text{exists in }\ L^p\big\}.
\end{equation}
One knows that $(\sA_p,D_p)$ is closed and densely defined (see \cite[Theorem 1.4]{EN}). 

\begin{corollary}\label{C:lp}
We have $\So\subset D_p$. Furthermore, $\sA_p f = \sA f$ for any $f\in \So$, and $\So$ is a core for $(\sA_p,D_p)$.
\end{corollary}

\begin{remark}\label{R:id}
From now on for a given $p\in [1,\infty]$ with a slight abuse of notation we write $\sA : D_p\to \Lp$ instead of $\sA_p$. In so doing, we must keep in mind that $\sA$ actually indicates the closed operator $\sA_p$ that, thanks to Corollary \ref{C:lp}, coincides with the differential operator $\sA$ on $\So$. Using this identification we will henceforth say that $(\sA,D_p)$ is the infinitesimal generator of the semigroup $\{P_t\}_{t>0}$ on $\Lp$.
\end{remark}

We omit the proof of the next lemma since it is a direct consequence of  (ii), (iv) in Lemma \ref{L:Pt}, and of \cite[Theorem 1.10]{EN}.

\begin{lemma}\label{L:specter}
Assume that \eqref{trace} be in force, and let $1\le p \le \infty$. Then: 
\begin{itemize}
\item[(1)] For any $\la\in \mathbb C$ such that $\Re \la >0$, we have $\la\in \rpp$;
\item[(2)] If $\la\in \mathbb C$ such that $\Re \la >0$, then $R(\la,\sA)$ exists and for any $f\in \Lp$ it is given by the formula $R(\la,\sA) f = \int_0^\infty e^{-\la t} P_t f\ dt$;
\item[(3)] For any $\Re \la > 0$ we have $||R(\la,\sA)||_{p\to p} \le \frac{1}{\Re \la}$.
\end{itemize}
\end{lemma}



\subsection{The nonlocal operators}\label{SS:fpA}

Since the fractional operators $\As$ play a central role in the present work we recall their definition from \cite{GT}. Hereafter, when considering the action of the operators $\sA$ or $\As$ on a given $\Lp$, the reader should keep in mind our Remark \ref{R:id}.

\begin{definition}\label{D:flheat}
Let $0<s<1$. For any $f\in \So$ we define the nonlocal operator $\As$ by the following pointwise formula
\begin{align}\label{As}
(-\mathscr A)^s f(X) & =  - \frac{s}{\G(1-s)} \int_0^\infty t^{-(1+s)} \left[P_t f(X) - f(X)\right] dt,\qquad X\in\RN.
\end{align}
\end{definition}

We mention that Definition \ref{D:flheat} comes from Balakrishnan's seminal work \cite{B}. It was shown in \cite{GT} that the right-hand side of \eqref{As} is a convergent integral (in the sense of Bochner) in $L^\infty$, and also in $L^p$ for any $p\in [1,\infty]$ when \eqref{trace} holds.  
The nonlocal operators \eqref{As} enjoy the following semigroup property established in \cite{B}.

\begin{proposition}\label{P:bala2}
Let $s, s'\in (0,1)$ and suppose that $s+s'\in (0,1]$. Then, for every $f\in \So$ we have
\[
(-\mathscr A)^{s+s'} f = (-\mathscr A)^s \circ (\mathscr A)^{s'} f.
\]
\end{proposition}

For any given $1\le p<\infty$, and any $0<s<1$, we denote by 
\begin{equation}\label{Ds}
D_{p,s} = \{f\in L^p\mid \As f \in L^p\},
\end{equation}
the domain of $\As$ in $L^p$. The operator $\As$ can be extended to a closed operator on its domain, see \cite[Lemma 2.1]{B}. Therefore, endowed with the graph norm
$$||f||_{D_{p,s}} \overset{def}{=} ||f||_{p} + ||(-\sA)^s f||_{p},$$
$D_{p,s}$ becomes a Banach space. 
The next lemma is proved in \cite[Lemma 4.3]{GThls} and it shows that, when \eqref{trace} holds, then
$\So\ \subset \ D_{p,s}$.

\begin{lemma}\label{L:inclusion}
Assume \eqref{trace}, and let $0<s<1$. Given $1\le p \le \infty$, one has 
\[
(-\sA)^s(\So) \subset L^p.
\]
\end{lemma}

We now use the nonlocal operators $\As$ to introduce the functional spaces naturally attached to the operator $\sA$.

\begin{definition}[Sobolev spaces]\label{D:sobolev}
Assume \eqref{trace}, and let $1\le p < \infty$ and $0<s<1$. We define the Sobolev space as $\Lo = \overline{\So}^{|| \  ||_{D_{p,s}}}$. 
\end{definition}

We next establish a key density result.

\begin{proposition}\label{density}
Assume \eqref{trace}. Let $0<s<1$ and $p\geq 1$. We have
$$D_{p,s}=\mathscr{L}^{2s,p}.$$
\end{proposition}
\begin{proof}
The inclusion $\mathscr{L}^{2s,p}\subseteq D_{p,s}$ is straightforward since by Lemma \ref{L:inclusion} we have $\So\subseteq D_{p,s}$, and by definition $\mathscr{L}^{2s,p}=\overline{\So}^{D_{p,s}}$. For the opposite inclusion $D_{p,s}\subseteq \mathscr{L}^{2s,p}$, we divide the proof into multiple steps.

\medskip

\noindent \emph{Step $I$}. For a fixed $\e>0$, consider the operator $\sA_\e=\sA - \e I: D_p \longrightarrow L^p$, where we have denoted by $D_p$ the domain of $\sA$ in $L^p$, see \eqref{Dp}, with its graph norm. In view of Lemma \ref{L:specter}, $\sA_\e$ is invertible, and the inverse is given by the formula $R(\e,\sA) f=\int_0^\infty e^{-\e t}P_t f dt$. Denote by $P^{(\e)}_t=e^{-\e t}P_t$ the semigroup having $\sA_\e$ as infinitesimal generator. Similarly to \eqref{As} in Definition \ref{D:flheat}, we can use Balakrishnan's formula to define the fractional powers 
\begin{equation}\label{Aes}
(-\mathscr A_\e)^s f(X)  =  - \frac{s}{\G(1-s)} \int_0^\infty t^{-(1+s)} \left[P^{(\e)}_t f(X) - f(X)\right] dt,\qquad X\in\RN.
\end{equation}
As in \cite[Theorem 6.3]{GThls}, one sees that the inverse of $\left(-\sA_\e\right)^{s}$ is given by
\[
\left(-\sA_\e\right)^{-s}f(X)=\frac{1}{\G(s)}\int_0^\infty t^{s-1} P^{(\e)}_t f(X)dt,\qquad X\in\RN.
\]
In view of (iv) in Lemma \ref{L:Pt}, it is easy to see that $\left(-\sA_\e\right)^{-s}$ is well-defined in $L^p$. We denote by $D^{(\e)}_{p,s}$ the domain of $\left(-\sA_\e\right)^s$ in $L^p$. We have $D_p\subseteq D^{(\e)}_{p,s}$, and moreover one easily concludes from \eqref{Aes} that there exists $C(s)>0$ such that for every $f\in D_p$ one has
\begin{equation}\label{piuforte}
 ||\left(-\sA_\e\right)^s f||_p\leq C(s) \left( ||f||_p + ||\sA_\e f||_p \right).
\end{equation}

\medskip

\noindent \emph{Step $II$}. We claim that $\So$ is dense in $D^{(\e)}_{p,s}$ with the graph norm, that is 
\[
\overline{\So}^{D^{(\e)}_{p,s}}=D^{(\e)}_{p,s}.
\]
To see this, we set $Z=\left(-\sA_\e\right)^s\left(D_p\right)$. By the invertibility of $\sA_\e$ we deduce that $\left(\sA_\e\right)^{1-s}Z=L^p$, and therefore $Z = D^{(\e)}_{p,1-s}$. On the other hand, $D^{(\e)}_{p,1-s}$ is dense in $L^p$ since we know that $D_p\subseteq D^{(\e)}_{p,1-s}$, and $D_p$ is dense in $L^p$. Therefore, $Z$ is dense in $L^p$, which implies that $D_p=\left(-\sA_\e\right)^{-s}Z$ is dense in $D^{(\e)}_{p,s}$ with the graph norm of $\left(-\sA_\e\right)^s$. Since we also know that $\overline{\So}^{D_p}=D_p$, by \eqref{piuforte} we  conclude that $\overline{\So}^{D^{(\e)}_{p,s}}=\overline{D_p}^{D^{(\e)}_{p,s}}=D^{(\e)}_{p,s}$.

\medskip

\noindent \emph{Step $III$}. We finish the proof by showing that
$$
\overline{\So}^{D_{p,s}}=D_{p,s}.
$$
By \cite[Lemma 4.11]{Lun} we know that $D^{(\e)}_{p,s}=D_{p,s}$, and there exists $C>0$ such that for every $f\in D_{p,s}$
$$||\left(-\sA_\e\right)^sf- \left(-\sA\right)^sf||_p\leq C \e^s ||f||_p.$$
Hence, for every $f\in D_{p,s}=D^\e_{p,s}$, we can find by {\em Step II} a sequence of functions $f_k\in \So$ such that both $||f_k-f||_p$ and $||\left(-\sA_\e\right)^s(f_k-f)||_p$ tend to $0$ as $k\rightarrow \infty$. Then, we also have
$$||\left(-\sA\right)^s(f_k-f)||_p\leq ||\left(-\sA_\e\right)^s(f_k-f)||_p + C \e^s ||f_k-f||_p \longrightarrow 0\quad\mbox{as }k\rightarrow \infty.$$
This concludes the proof of the density of $\So$ in $D_{p,s}$ endowed with its graph norm.

\end{proof}



\section{Indicator functions}\label{S:ind}

In this section we establish some results concerning indicator functions of measurable sets $E\subset \RN$. Some of our results are inspired to the ideas of M. Ledoux in \cite{Led}.

\begin{lemma}\label{L:charheat}
If $E\subset \RN$ is a measurable set with finite measure, one has
\[
||P_t \mathbf 1_E - \mathbf 1_E||_{1} = (1+e^{-t \operatorname{tr} B}) |E| -  2 \int_E P_t \mathbf 1_E(X) dX.
\]
\end{lemma}

\begin{proof}
From (i) in Lemma \ref{L:Pt} we find
\[
1 = \int_E p(X,Y,t) dY + \int_{\RN\setminus E} p(X,Y,t) dY.
\]
Integrating in $X\in E$ we find
\[
|E| = \int_E \int_E p(X,Y,t) dY dX + \int_E \int_{\RN\setminus E} p(X,Y,t) dY dX.
\]
Similarly, (iii) in Lemma \ref{L:Pt} gives
\[
e^{-t \operatorname{tr} B} = \int_E p(X,Y,t) dX + \int_{\RN\setminus E} p(X,Y,t) dX.
\]
Integrating this identity in $Y\in E$, we find
\[
e^{-t \operatorname{tr} B} |E| = \int_E \int_E p(X,Y,t) dX dY + \int_E \int_{\RN\setminus E} p(X,Y,t) dX dY.
\]
This gives
\begin{align*}
 |E| + e^{-t \operatorname{tr} B} |E| & = \int_E \int_E p(X,Y,t) dY dX + \int_E \int_{\RN\setminus E} p(X,Y,t) dY dX 
\\
& + \int_E \int_E p(X,Y,t) dX dY + \int_E \int_{\RN\setminus E} p(X,Y,t) dX dY
\\
& = 2 \int_E P_t \In(X) dX + \int_E \int_{\RN\setminus E} p(X,Y,t) dY dX + \int_E \int_{\RN\setminus E} p(X,Y,t) dX dY.
\end{align*}
To reach the desired conclusion we are thus left with showing that
\[
||P_t \mathbf 1_E - \mathbf 1_E||_{1} = \int_E \int_{\RN\setminus E} p(X,Y,t) dY dX + \int_E \int_{\RN\setminus E} p(X,Y,t) dX dY.
\]
This is easily verified as follows
\begin{align*}
||P_t \mathbf 1_E - \mathbf 1_E||_{1} & = \int_{\RN} \left|\mathbf 1_E(X) - P_t \mathbf 1_E(X)\right| dX
\\
& = \int_E \left|1 - \int_E p(X,Y,t) dY \right| dX + \int_{\RN\setminus E}  \int_E p(X,Y,t) dY  dX
\\
& = \int_E \left(1 - \int_{E} p(X,Y,t)   dY\right) dX +  \int_E  \int_{\RN\setminus E}  p(X,Y,t)  dX dY 
\\
& = \int_E \int_{\RN\setminus E} p(X,Y,t) dY dX + \int_E \int_{\RN\setminus E} p(X,Y,t) dX dY.
\end{align*}

\end{proof}

The following lemma provides a crucial estimate from below of the right-hand side of Lemma \ref{L:charheat}.

\begin{lemma}\label{L:perbelow}
There exists a universal constant $b_N>0$ such that for any measurable set $E\subset \RN$ such that $|E|<\infty$, one has
\[
||P_t \mathbf 1_E - \mathbf 1_E||_{1} \ge |E| -  \frac{b_N}{V(t/2)} e^{-\frac t4 \operatorname{tr} B} |E|^2.
\]
\end{lemma}

\begin{proof}
Our objective is estimating the terms in the right-hand side of Lemma \ref{L:charheat}. Let $E\subset \RN$ be a measurable set such that $|E|<\infty$, then by (v) in Lemma \ref{L:Pt} one has
\begin{equation}\label{PtE}
\int_{E} P_t \mathbf 1_E(X) dX = (P_t \mathbf 1_E,\mathbf 1_E) = (P_{t/2} \mathbf 1_E, P^\star_{t/2} \mathbf 1_E) \le ||P_{t/2} \mathbf 1_E||_2 ||P^\star_{t/2} \mathbf 1_E||_2. 
\end{equation}
By Minkowski's integral inequality we have 
\begin{align*}
& ||P_{t/2} \mathbf 1_E||_2 = \left(\int_{\RN} \left(\int_E p(X,Y,t/2) dY\right)^2 dX\right)^{1/2} \le \int_E \left(\int_{\RN} p(X,Y,t/2)^2 dX\right)^{1/2} dY.
\end{align*} 
Before computing the integral in the right-hand side of the latter inequality, we notice that if  in $\RN$  we make the change of variable $Z = \frac{K(t)^{-1/2}(Y-e^{tB} X)}{\sqrt{t}}$, then 
\[
dZ = \frac{(\det K(t))^{-1/2}}{t^{N/2}} e^{t \operatorname{tr} B} dX = \omega_N \frac{e^{t \operatorname{tr} B}}{V(t)} dX,
\]
where in the last equality we have used \eqref{VS}. By \eqref{p} and \eqref{m} we thus find for every $t>0$
\begin{align}\label{psquare}
\int_{\RN} p(X,Y,t)^2 dX & = \frac{c_N^2}{V(t)^2} \int_{\RN} \exp\left(-\frac{|K(t)^{-1/2}(Y-e^{tB} X)|^2}{2t}\right) dX
\\
& = \frac{c_N^2 \omega_N^{-1} e^{- t \operatorname{tr} B}}{V(t)} \int_{\RN} e^{-\frac{|Z|^2}{2}} dZ = \frac{a_N e^{- t \operatorname{tr} B}}{V(t)}.
\notag
\end{align}
Using \eqref{psquare} in the above inequality, we obtain 
\[
||P_{t/2} \mathbf 1_E||_2 \le \frac{a_N^{1/2}}{V(t/2)^{1/2}} e^{-\frac t4 \operatorname{tr} B} |E|.
\]
In a similar fashion we find 
\[
||P^\star_{t/2} \mathbf 1_E||_2 \le \frac{a_N^{1/2}}{V(t/2)^{1/2}} |E|.
\]
Using the latter two estimates in \eqref{PtE} we conclude
\begin{equation}\label{ptsquare}
\int_{E} P_t \mathbf 1_E(X) dX \le \frac{a_N}{V(t/2)} e^{-\frac t4 \operatorname{tr} B} |E|^2.
\end{equation}
Combining Lemma \ref{L:charheat} with \eqref{ptsquare} we reach the desired conclusion with $b_N = 2 a_N$.

\end{proof}

In what follows the reader should keep in mind definition \eqref{Ds}. The next lemma is contained in \cite[Prop. 3.4 \& Remark 3.5]{GTfi}.

\begin{lemma}\label{L:alteAs}
Let $s\in (0,1)$ and $E\subset \RN$ be a measurable set such that $\mathbf 1_E\in D_{1,s}$. Then, 
\[
||(-\sA)^s \In||_{1} =  \frac{s}{\G(1-s)} \int_0^\infty \frac{1}{t^{1+s}} ||P_t\In - \In||_{1} dt.
\]
\end{lemma}

\begin{lemma}\label{L:comm}
Assume \eqref{trace} and let $s\in (0,1)$. For any $f\in L^1$ such that
$$\int_0^\infty\frac{||P_\tau f - f||_1}{\tau^{1+s}}d\tau<+\infty,$$
we have $f, P_t f\in D_{1,s}$ for all $t>0$, and moreover
\begin{equation}\label{commutano}
(-\sA)^s P_tf=P_t(-\sA)^sf
\end{equation}
a. e. in $\RN$.
In particular, \eqref{commutano} holds true for any $f\in\So$.
\end{lemma}

\begin{proof}
The fact that $f\in D_{1,s}$ follows since $X\to (-\sA)^sf(X)$ is measurable and we can write by \eqref{As}
\begin{align*}
\int_{\RN}|(-\sA)^sf(X)|dX & \leq \frac{s}{\G(1-s)} \int_{\RN}\int_0^\infty\frac{|P_\tau f(X) - f(X)|}{\tau^{1+s}}d\tau dX
\\
& = \frac{s}{\G(1-s)} \int_0^\infty\frac{||P_\tau f - f||_1}{\tau^{1+s}}d\tau<+\infty.
\end{align*}
Now, when \eqref{trace} holds, then by (iii) and (iv) in Lemma \ref{L:Pt} we have $||P_\tau P_t f-P_tf||_1\leq ||P_\tau f - f||_1$, for any $t>0$. Therefore, 
\[
\int_0^\infty\frac{||P_\tau P_t f - P_t f||_1}{\tau^{1+s}}d\tau<+\infty,
\]
and by the first part of the lemma we infer that $P_t f\in D_{1,s}$. Combining Fubini's and Tonelli's theorems we obtain for almost every $X\in \RN$,  
\begin{align*}
&(-\sA)^s P_tf(X)= - \frac{s}{\G(1-s)}\int_0^\infty \frac{P_t\left(P_\tau f - f\right)(X)}{\tau^{1+s}}d\tau \\
&= - \frac{s}{\G(1-s)}\int_0^\infty\int_{\RN} p(X,Y,t)\frac{P_\tau f(Y)- f(Y)}{\tau^{1+s}}dYd\tau\\
&= - \frac{s}{\G(1-s)}\int_{\RN} p(X,Y,t)\left(\int_0^\infty \frac{P_\tau f(Y)- f(Y)}{\tau^{1+s}}d\tau \right)dY=P_t\left((-\sA)^sf\right)(X),
\end{align*}
for all $t>0$. Finally, the statement regarding $f\in\So$ follows from Lemma \ref{L:inclusion}.

\end{proof}

\begin{corollary}\label{corinfondo}
Under the hypothesis of Lemma \ref{L:comm}, suppose $\In\in D_{1,s}$. Then,
$$\lim_{t\rightarrow 0^+}||(-\sA)^s P_t\In||_1=||(-\sA)^s \In||_1.$$
\end{corollary}

\begin{proof}
By Lemma \ref{L:alteAs} we have
$$\frac{s}{\G(1-s)} \int_0^\infty\frac{||P_\tau \In - \In||_1}{\tau^{1+s}}d\tau=||(-\sA)^s \In||_1<\infty.$$
Then, we can use \eqref{commutano} in Lemma \ref{L:comm} and Corollary \ref{C:Ptpzero} to deduce that
\[
(-\sA)^s P_t \In=P_t(-\sA)^s\In\ \underset{t\to 0^+}{\longrightarrow}\ (-\sA)^s\In\]
in $L^1$. This proves the statement.
\end{proof}


\section{Nonlocal perimeters}\label{S:per}

In this section we introduce the notions of function of nonlocal bounded variation and that of fractional perimeter. Our approach, based on the interplay between the semigroup $\{P_t\}_{t>0}$ and the fractional powers $\As$, provides a different perspective to nonlocal interactions and minimal surfaces even in the classical (non-degenerate) case when there is no drift and $\sA = \Delta$. As we have said in the introduction, the operators $\sA$ in \eqref{A0} do not possess a variational structure, i.e., they do not arise as Euler-Lagrange equations of an energy. To bypass this obstacle, we use instead a relaxation procedure.


\subsection{Functions of nonlocal bounded variation}\label{SS:bvs}

Given a function $f\in L^1$ we consider the family $\mathscr F(f) = \{\{f_k\}_{k\in \mathbb N}\subset \So\mid f_k \to f\ \text{in}\  L^1\}$. Obviously, we have $\mathscr F(f) \not= \varnothing$. In the sequel we tacitly use the fact that, if $f\in \So$, then $\As f\in L^1$. This property is guaranteed by Lemma \ref{L:inclusion}.

\begin{definition}[The space $BV^{\sA}_s$]\label{D:BV}
Given $0<s<1/2 $, we say that $f\in L^1$ has \emph{bounded $s$-variation} if there exists $\{f_k\}_{k\in \mathbb N}$ in $\mathscr F(f)$ such that
\[
{\bf{V}}^{\sA}_s(f;\{f_k\}) \overset{def}{=} \underset{k\to \infty}{\liminf}\  ||\As f_k||_1 < \infty.
\]
When $f$ has bounded $s$-variation we call the $s$-\emph{total variation} of $f$ the number in $[0,\infty)$ defined by
$${\bf{V}}^{\sA}_s(f)  = \underset{\{f_k\}_{k\in \mathbb N}\in \mathscr F(f)}{\inf}\ {\bf{V}}^{\sA}_s(f;\{f_k\}).$$
Given a measurable set $E\subset \RN$, with $|E|<\infty$, we say that it has finite $s$-\emph{perimeter} if $\mathbf 1_E \in BV^\sA_s$. In such case we define the $s$-perimeter of $E$ as the number in $[0,\infty)$ identified by
\begin{equation}\label{sperimeter}
\Ps(E) \overset{def}{=} {\bf{V}}^\sA_s(\mathbf 1_E) = \underset{\{f_k\}_{k\in \mathbb N}\in \mathscr F(\mathbf 1_E)}{\inf}\ \underset{k\to \infty}{\liminf}\ ||\As f_k||_1.
\end{equation}
\end{definition}

We mention that in the local case M. Miranda, jr. used a relaxation procedure to define a perimeter \`a la De Giorgi in metric measure spaces which are doubling and Poincar\'e, see \cite{Mi}.

\begin{remark}\label{R:s}
Since we have defined the nonlocal operators $\As$ for all $s\in (0,1)$, the reader may wonder why in Definition \ref{D:BV} we are restricting the range to $s\in (0,1/2)$. The full explanation for this will come after we prove in Proposition \ref{P:equalstars} that, when $\sA = \Delta$, our perimeter $\Ps$ equals that of Caffarelli, Roquejoffre and Savin. This result, combined with the fact that, when $1/2\le s<1$, no open set has finite nonlocal $s$-perimeter (see e.g. \cite{Sickel}, or the explicit constant in \cite[Prop. 1.1]{GAL}), clarifies the limitation $0<s<1/2$. 
\end{remark}

\begin{lemma}\label{L:perAs}
Let $E\subset \RN$ be a measurable set such that $\mathbf 1_E\in D_{1,s}$. Then,
\[
||\As \mathbf 1_E||_1 \ge \Ps(E).
\]
\end{lemma}

\begin{proof}
By Proposition \ref{density} we know that we can approximate $\mathbf 1_{E}\in D_{1,s}=\mathscr{L}^{2s,1}$ with a sequence of functions $g_k\in\So$ with $g_k$ tending to $\mathbf 1_{E}$ in the graph norm of $\As$. Hence, we obtain
$$
\mathfrak P^{\sA}_s\left(E\right)\leq \underset{k\to \infty}{\lim}\| \As g_k \|_{L^1(\RN)} =\| \As \mathbf 1_{E} \|_{L^1(\RN)}.
$$

\end{proof}


\subsection{Another notion of perimeter}\label{SS:another}

We next introduce a second notion of nonlocal perimeter which is inspired to De Giorgi's original one in the local case and which will prove useful in our analysis.
We notice preliminarily that, if $E\subset\RN$ is a bounded measurable  set, then $P_t\mathbf 1_E$ belongs to $\So$ for all $t>0$. This can be recognised by showing that, equivalently, $\widehat{P_t\mathbf 1}_E \in \So$ for all $t>0$. To see this latter fact, we use \cite[formula (2.6)]{GT} which gives 
\[
\widehat{P_t\mathbf 1}_E = e^{-t \operatorname{tr} B}  e^{- 4 \pi^2 <C(t)\xi,\xi>} \hat{\mathbf 1}_E(e^{-tB^\star} \xi).
\]
Here, $C(t)$ is the positive symmetric matrix-valued function defined by  
$t K(t) = e^{tB} C(t) e^{tB^\star}$,
see \eqref{Kt}.
Since $\xi \to e^{- 4 \pi^2 <C(t)\xi,\xi>}$ belongs to $\So$, it suffices to show that $\xi \to \hat{\mathbf 1}_E(e^{-tB^\star} \xi)$ is a multiplier for $\So$. From the boundedness of $E$ it is easy to observe that $\xi \to \hat{\mathbf 1}_E(\xi)$ belongs to $C^\infty(\RN)$, and thus such is also $\xi \to \hat{\mathbf 1}_E(e^{-tB^\star} \xi)$. Furthermore, any derivative of this latter function grows at most polynomially. This proves that $\widehat{P_t\mathbf 1}_E$, and therefore $P_t\mathbf 1_E$, belongs to $\So$. By Lemma \ref{L:inclusion}, we can thus consider $\As P_t\mathbf 1_E$.
Arguing as in the proof of Lemma \ref{L:pdec} we see that the function $t\to \left\|\As P_t\mathbf 1_E\right\|_1$ is monotone decreasing on $(0,\infty)$. This allows us to introduce the following second definition of nonlocal perimeter.

\begin{definition}\label{per2}
Given $0<s<1/2$, we say that a bounded measurable set $E\subset \RN$, has finite $s$-\emph{perimeter$^\star$} if
\begin{equation}\label{sstar}
\mathfrak P^{\sA,\star}_s(E) \overset{def}{=} \underset{t\to 0^+}{\lim}\  ||\As P_t \mathbf 1_E||_1=\underset{t> 0}{\sup}\ ||\As P_t \mathbf 1_E||_1 < \infty.
\end{equation}
In this case, we call the number $\mathfrak P^{\sA,\star}_s(E)\in [0,\infty)$ the  $s$-perimeter$^\star$ of $E$.
\end{definition}

It is interesting to compare the two nonlocal perimeters \eqref{sperimeter} and \eqref{sstar}. The following simple result holds.

\begin{lemma}\label{L:senzastarminoredistar}
For every $s\in (0,1/2)$ and every bounded measurable set $E\subset \RN$ we have
\[
\mathfrak P^\sA_s(E) \le \mathfrak P^{\sA,\star}_s(E).
\]
\end{lemma}

\begin{proof}
Fix a bounded measurable set $E\subset \RN$. If $\mathfrak P^{\sA,\star}_s(E) = \infty$ there is nothing to prove. We thus assume that $\mathfrak P^{\sA,\star}_s(E)<\infty$. Consider the sequence $t_k  = k^{-1} \searrow 0$. As we have observed, $f_k = P_{t_k}\mathbf 1_E\in \So$ for every $k\in \mathbb N$. Moreover, one has
\[
\underset{k\to \infty}{\liminf}\ ||\As f_k||_1 \le \underset{t\to 0^+}{\lim}\ ||\As P_t \mathbf 1_E||_1 = \mathfrak P^{\sA,\star}_s(E) < \infty.
\]
Since by Corollary \ref{C:Ptpzero} we know that $f_k = P_{t_k} \mathbf 1_E \to \mathbf 1_E$ in $L^1$, we infer that 
$\{f_k\}_{k\in \mathbb N}\in \mathscr F(\mathbf 1_E)$ and that 
\[
\underset{k\to \infty}{\liminf}\ ||\As f_k||_1 \le \mathfrak P^{\sA,\star}_s(E).
\]
By the definition \eqref{sperimeter} we immediately reach the desired conclusion $\Ps(E) \le \mathfrak P^{\sA,\star}_s(E)$. 

\end{proof}

\begin{remark}\label{R:percomp}
We presently do not know whether $\Ps(E) = \mathfrak P^{\sA,\star}_s(E)$. However, the reader should see Proposition \ref{P:equalstars} below.
\end{remark}

\subsection{Connection with the nonlocal perimeter of Caffarelli, Roquejoffre and Savin}

It is natural at this moment to compare, when $\sA = \Delta$ in \eqref{A0}, our $\Ps$ and $\mathfrak P^{\sA,\star}_s$ to the notion of nonlocal perimeter (implicitly present) in the works by Almgren and Lieb \cite{AL}, Bourgain, Brezis and Mironescu \cite{BBM1,BBM2}, and Maz'ya \cite{Ma1}, and extensively developed by Caffarelli, Roquejoffre and Savin in \cite{CRS}. We recall that these latter authors say that when $0<s<1/2$ a measurable set $E\subset \RN$ has finite $s$-perimeter if
\begin{align}\label{crs}
P_s(E)  \overset{def}{=} &  \int_{\RN}\int_{\RN} \frac{|\mathbf 1_E(X) - \mathbf 1_E(Y)|^2}{|X-Y|^{N+2s}} dX dY  = 2 \int_{\RN\setminus E}\int_{E} \frac{dX dY}{|X-Y|^{N+2s}} < \infty.
\end{align}

The next proposition shows that, specialised to the classical setting, the quantities $\Ps, \mathfrak P^{\sA,\star}_s$ coincide, and they equal $P_s$ up to a constant. 

\begin{proposition}\label{P:equalstars}
Suppose that $\sA = \Delta$ in \eqref{A0}, and therefore $P_t$ is the standard heat semigroup. Let $0<s<1/2$. For any bounded measurable set $E\subset \RN$ such that $P_s(E)<\infty$, we have
\[
\mathfrak P^\Delta_s(E) = \mathfrak P^{\Delta,\star}_s(E) = c^\star(N,s)\ P_s(E),
\]
where $c^\star(N,s)>0$ is an explicit constant. 
\end{proposition}

\begin{proof}
We first prove the equality $\mathfrak P^{\Delta,\star}_s(E) = c^\star(N,s) P_s(E)$. By Corollary \ref{corinfondo} we know that $P_s(E)<\infty$ is equivalent to requesting 
that $||(-\Delta)^s \In||_{1} < \infty$. Since the assumption \eqref{crs} can be equivalently expressed as 
\[
\int_0^\infty\frac{||P_\tau \In - \In||_1}{\tau^{1+s}}d\tau<+\infty,
\]
we can use the commutation identity \eqref{commutano} in Lemma \ref{L:comm}. Combined with Corollary \ref{C:Ptpzero} this guarantees that 
\[
||(-\Delta)^s P_t \In||_1 = ||P_t (-\Delta)^s \In||_1 \ \underset{t\to 0^+}{\longrightarrow}\ ||(-\Delta)^s \In||_1.
\]
Having in mind \cite[Corollary 3.6]{GTfi}, we conclude that 
\[
\mathfrak P^{\Delta,\star}_s(E) = \underset{t\to 0^+}{\lim} ||(-\Delta)^s P_t \In||_1 = ||(-\Delta)^s \In||_1 = c^\star(N,s)\ P_s(E),
\]
thus our $s$-perimeter$^\star$ coincides (up to an absolute constant) with that in  \cite{CRS}. In view of Lemma \ref{L:senzastarminoredistar}, to complete the proof we are left with showing that
\begin{equation}\label{reversethestars}
\mathfrak P^{\Delta,\star}_s(E) \le \mathfrak P^{\Delta}_s(E).
\end{equation}  
It is not restrictive to suppose that $\mathfrak P^{\Delta}_s(E)<\infty$. By \eqref{sperimeter} for any $\ve>0$ there exists a sequence of functions $\{f_k\}_{k\in \mathbb N}= \{f_k^{(\ve)}\}_{k\in \mathbb N}$ in $\mathscr F(\mathbf 1_E)$ such that
\[
\liminf_{k\rightarrow \infty}||(-\Delta)^s f_k||_1\leq \mathfrak P^{\Delta}_s(E) + \ve.
\]
For any $t>0$ consider $P_t f_k \in \So$. From (iv) in Lemma \ref{L:Pt}  and \eqref{commutano}, we find
\begin{equation*}
||(-\Delta)^s P_t f_k||_1=||P_t(-\Delta)^s f_k||_1\leq ||(-\Delta)^s f_k||_1\qquad\forall k\in\N,\, t>0.
\end{equation*}
Passing to the $\liminf$ as $k\to \infty$ in the latter inequality gives
\begin{equation}\label{meglioPtfk2E}
\underset{k\to \infty}{\liminf}\ ||(-\Delta)^s P_t f_k||_1 \leq \mathfrak P^{\Delta}_s(E) + \ve.
\end{equation}
We claim that for any fixed $t>0$,
\begin{equation}\label{ohmanipadmeumE}
||(-\Delta)^s P_t 1_E||_1 \le \underset{k\to \infty}{\liminf}\ ||(-\Delta)^s P_t f_k||_1.
\end{equation}
Taking the claim for granted, from it and \eqref{meglioPtfk2E}
 we obtain
\[
||(-\Delta)^s P_t 1_E||_1 \le \mathfrak P^{\Delta}_s(E) + \ve.
\]
Since this holds for every $t>0$, we infer
\[
\mathfrak P^{\Delta,\star}_s(E) = \underset{t> 0}{\sup}\ ||(-\Delta)^s P_t \mathbf 1_E||_1 \leq \mathfrak P^{\Delta}_s(E) + \ve.
\]
By the arbitrariness of $\ve$ we reach the desired conclusion \eqref{reversethestars}. We are thus left with proving \eqref{ohmanipadmeumE}. Since we obviously have for every $k\in \mathbb N$
\[
||(-\Delta)^s P_t 1_E||_1 \le ||(-\Delta)^s P_t f_k||_1 + ||(-\Delta)^s (P_t 1_E - P_t f_k)||_1,
\]
the claim \eqref{ohmanipadmeumE} will follow if we show that
\begin{equation}\label{ohmanipadmeumEE}
\underset{k\to \infty}{\liminf}\ ||(-\Delta)^s (P_t 1_E - P_t f_k)||_1 = 0.
\end{equation}
Now, since we know that $||(-\sA)^s f||_p\leq c\left(||\sA f||_p + ||f||_p\right)$ for all $f\in\So$ (see the proof of  \cite[Lemma 4.3]{GThls}), we obtain
$$||(-\Delta)^s \left(P_t f_k- P_t 1_E\right)||_1\leq C_s \left(||P_t f_k- P_t 1_E||_1 + ||\Delta\left(P_t f_k- P_t 1_E\right)||_1\right),$$
for some $C_s>0$. On the one hand we have
$$||P_t f_k- P_t 1_E||_1\leq ||f_k- 1_E||_1\rightarrow 0\quad\mbox{ as }k\rightarrow+\infty$$
since $f_k\in \mathscr F(\mathbf 1_E)$. On the other hand, it is easy to verify that 
\begin{align*}
& ||\Delta P_t\left(f_k- 1_E\right)||_1\leq \sup_{Y\in\RN}\int_{\RN}|\Delta p(X,Y,t)|dX\ ||f_k- 1_E||_1
\\
& \leq c(t)\ ||f_k- 1_E||_1\ \underset{k\to \infty}{\longrightarrow}\ 0.
\end{align*}
 This proves that, for any $t>0$,
$$||(-\Delta)^s\left(P_t f_k- P_t 1_E\right)||_1\ \underset{k\to \infty}{\longrightarrow}\ 0,$$
which implies \eqref{ohmanipadmeumEE}.

\end{proof}



\section{A key nonlocal estimate of Ledoux type}\label{S:ledoux}

In \cite{CLY} Cheng, Li and Yau first showed that the isoperimetric inequality implies upper bounds for the heat kernel on a complete manifold satisfying various curvature assumptions. Shortly after, in his approach to the Hardy-Littlewood-Sobolev inequalities, Varopoulos used the heat semigroup to connect analysis to geometry. One of his central results states that the ultracontractive estimate 
\[
||e^{-t\Delta} f||_\infty \le \frac{C}{t^{n/2}} ||f||_1
\]
is equivalent to the $L^2$ Sobolev inequality \cite[Theor.1]{V85}. This remarkable result links upper estimates of the heat semigroup to Sobolev inequalities at level $p=2$, and eventually to isoperimetry, i.e., Sobolev inequalities at level $p=1$. Inspired by Varopoulos' works, Ledoux showed  how to reverse these ideas and obtain isoperimetric inequalities from upper bounds of the heat semigroup. In his approach one of the key tools was the following estimate 
\begin{equation}\label{led}
||e^{-t\Delta} f - f||_1 \le C \sqrt t ||\nabla f||_1,
\end{equation}
valid for any function $f\in C^\infty_0(M)$, where $M$ is a Riemannian manifold with $\operatorname{Ric}\ge 0$,
see \cite{Led}, and also the proof of Theorem 8.14 in Ledoux's article in \cite{DGL}, in particular equation (8.14).

The objective of this section is to establish the following nonlocal version of \eqref{led} for the H\"ormander semigroup generated by $\sA$ in \eqref{A0}.

\begin{theorem}[Nonlocal Ledoux type estimate]\label{T:generaled}
Assume \eqref{trace}, and let $0<s<1$. Then, for every $f\in \So$ and every $t, \tau>0$ we have for any $1\le p<\infty$
\[
||P_t f - P_\tau f||_{p} \le \frac{2 |t-\tau|^s}{\G(1+s)}\ \underset{\sigma>0}{\sup}\ ||\As P_\sigma f||_p.
\]
\end{theorem} 
 
We remark that when $\sA = \Delta$ in \eqref{A0}, and therefore $P_t = e^{-t \Delta}$ is the standard heat semigroup in $\RN$, taking $p = 1$ in the statement of Theorem \ref{T:generaled} and $s = 1/2$ we obtain, letting $\tau \to 0^+$,
\[
||e^{-t\Delta} f - f||_{1} \le \frac{4 \sqrt t}{\sqrt \pi}\ \underset{\sigma>0}{\sup}\ ||\sqrt{-\Delta} P_\sigma f||_1 = \frac{4 \sqrt t}{\sqrt \pi}\ ||\sqrt{-\Delta} f||_1,
\]
where in the last equality we have used \eqref{commutano}, which gives $\sqrt{-\Delta} P_\sigma f = P_\sigma \sqrt{-\Delta} f$.
This estimate differs from \eqref{led} since the terms in the right-hand side are not comparable. One should keep in mind that the Riesz transforms do not map $L^1$ to itself, but into $L^{1,\infty}$.

The proof of Theorem \ref{T:generaled} will be given at the end of the section. In what follows we establish the main estimate in such proof. We begin 
with a key observation. 

\begin{lemma}\label{L:pdec}
Assume \eqref{trace}. Let $s\in (0,1)$ and $1\le p<\infty$ be arbitrarily fixed. Given $f\in \So$, the function $t\to ||\As P_t f||_p$ is non-increasing on $(0,\infty)$. Furthermore, 
\[
\underset{t\to 0^+}{\lim} ||\As P_t f||_p = \underset{t>0}{\sup} ||\As P_t f||_p < \infty.
\]
\end{lemma}

\begin{proof}
Suppose that $t>\tau$ and write $t = \sigma + \tau$. By (v) of Lemma \ref{L:Pt} and (c) in Lemma \ref{L:invS}, we have $\As P_t f = P_\sigma \As P_\tau f$. Combined with (iv) in Lemma \ref{L:Pt}, we infer
\[
||\As P_t f||_p = ||P_\sigma \As P_\tau f||_p \le ||\As P_\tau f||_p = ||P_\tau \As f||_p \le ||\As f||_p < \infty,
\]
where in the last step we have used Lemma \ref{L:inclusion}. This chain of inequalities proves the desired conclusion. 

\end{proof}

We next recall the Riesz potentials introduced in \cite{GThls}.

\begin{definition}\label{D:fi}
Let $0< \alpha < D_\infty$. Given $f\in \So$, we define the \emph{Riesz potential} of order $\alpha$ as follows
\[
\Ia f(X) = \frac{1}{\G(\alpha/2)} \int_0^\infty t^{\alpha/2 - 1} P_t f(X) dt.
\]
\end{definition}

The next result, which is \cite[Theor. 6.3]{GThls}, shows that the integral operator $\Ia$ is the inverse of the nonlocal operator $(-\sA)^{\alpha/2}$. The reader should keep in mind here that, as we have recalled in the introduction, $D_\infty\geq 2$.

\begin{theorem}\label{T:inverse}
Suppose that \eqref{trace} hold, and let $0<\alpha<2$. Then, for any $f\in \So$ we have
\[
f = \mathscr I_{\alpha} \circ (-\sA)^{\alpha/2} f = (-\sA)^{\alpha/2} \circ \mathscr I_{\alpha} f.
\]
\end{theorem}

Using Theorem \ref{T:inverse} we now establish the key step in the proof of Theorem \ref{T:generaled}. 
 
\begin{proposition}\label{P:ledoux}
Assume \eqref{trace} and let $0<s<1$. Then, for every $f\in \So$ and $t, \tau>0$ we have
\begin{equation}\label{ledoux}
P_t f(X) - P_\tau f(X) = \int_0^\infty \ell_s(\sigma;t,\tau) \As P_\sigma f(X) d\sigma,
\end{equation}
where the function $\ell_s(\sigma;t,\tau)$ is supported in $[\min\{\tau,t\},\infty)$ and has the property
\begin{equation}\label{ledpic}
\int_0^\infty |\ell_s(\sigma;t,\tau)| d\sigma = \frac{2 |t-\tau|^s}{\G(1+s)}.
\end{equation}
\end{proposition} 

\begin{proof}
By Theorem \ref{T:inverse} and Definition \ref{D:fi} we have 
\begin{align*}
P_t f - P_\tau f & = P_t(\mathscr I_{2s} \As f ) - P_\tau(\mathscr I_{2s} \As f) = \frac{\As}{\G(s)} \int_0^\infty \sigma^{s - 1}\left(P_{t + \sigma} f - P_{\tau +\sigma}f\right) d\sigma
\\
& = \frac{\As}{\G(s)} \left\{\int_t^\infty (\sigma-t)^{s - 1}P_{\sigma} f - \int_\tau^\infty (\sigma-\tau)^{s - 1} P_\sigma f d\sigma\right\}
\\
& = \int_0^\infty \ell_s(\sigma;t,\tau) \As P_\sigma f(X) d\sigma,
\end{align*}
where we have set
\[
\ell_s(\sigma;t,\tau) = \frac{1}{\G(s)} \left[\mathbf 1_{(t,\infty)}(\sigma)(\sigma-t)^{s-1} - \mathbf 1_{(\tau,\infty)}(\sigma)(\sigma-\tau)^{s-1}\right].
\]
We warn the reader that the validity of the above chain of identities is guaranteed by the short time behaviour of the semigroup in Lemma \ref{L:Lprate}, and also by its long time behaviour ensured by \eqref{uc}, \eqref{trace} and the ensuing blowup of $V(t)$ established in \cite{GThls} (see (i) and (ii) in the introduction).
It is now a simple calculus exercise to show that
\[
\int_0^\infty |\ell(\sigma;t,\tau)| d\sigma = \int_0^\infty \left|\mathbf 1_{(t,\infty)}(\sigma)(\sigma-t)^{s-1} - \mathbf 1_{(\tau,\infty)}(\sigma)(\sigma-\tau)^{s-1}\right| d\sigma = \frac{2 |t-\tau|^s}{s\G(s)},
\]
which proves \eqref{ledpic}, and therefore \eqref{ledoux}.

\end{proof}

We mention that in \cite{ABT} the authors obtain a result similar to Proposition \ref{P:ledoux} (but only for the case $s = 1/2$), for semigroups with self-adjoint positive generators. Their proof uses the spectral resolution of such operators and does not apply to our situation.

\begin{proof}[Proof of Theorem \ref{T:generaled}]
It is a direct consequence of Proposition \ref{P:ledoux} and Minkowski integral inequality. Indeed, one has
\begin{align*}
 ||P_t f - P_\tau f||_{p} & \le \int_0^\infty |\ell_s(\sigma;t,\tau)|\ ||\As P_\sigma f||_p d\sigma
 \\
 & \le \underset{\sigma>0}{\sup}\ ||\As P_\sigma f||_p \int_0^\infty |\ell_s(\sigma;t,\tau)| d\sigma = \frac{2 |t-\tau|^s}{\G(1+s)}\ \underset{\sigma>0}{\sup}\ ||\As P_\sigma f||_p.
 \end{align*}

\end{proof}

We also record the following immediate consequence of Theorem \ref{T:generaled}. 

\begin{corollary}\label{C:led}
Let $0<s<1$. For every $f\in \So$ one has
\[
\underset{t>0}{\sup}\ \frac{1}{t^s}\ ||P_t f - f||_{1} \le \frac{2}{\G(1+s)}\ \underset{\sigma>0}{\sup}\ ||\As P_\sigma f||_{1}.
\]
\end{corollary}


\section{Proof of Theorems \ref{T:iso01} and \ref{T:isoother0}}\label{S:main}

In this section we provide the proofs of our two main results. We begin with the

\begin{proof}[Proof of Theorem \ref{T:iso01}]
Since the desired conclusion is trivially true if $\Ps(E) = \infty$, we assume that $\Ps(E) <\infty$. By the hypothesis, $\mathbf 1_E\in L^1(\RN)$. Therefore, there exists a sequence $\{f_k\}_{k\in \mathbb N}$ in $\So$ such that $f_k\to \mathbf 1_E$ in $L^1(\RN)$ as $k\to \infty$. 
By (iv) in Lemma \ref{L:Pt} we know that $||P_t f_k - P_t \mathbf 1_E||_1\ \to\ 0$, and thus also $||P_t f_k - f_k||_1\ \to\ ||P_t \mathbf 1_E - \mathbf 1_E||_1$ as $k\to \infty$. By Theorem \ref{T:generaled} and Corollary \ref{C:led} we have for every $k\in \mathbb N$
\begin{align*}
||P_t f_k - f_k||_1 & \le \frac{2 t^s}{\G(1+s)} \ \underset{\sigma >0}{\sup}\ ||\As P_\sigma f_k||_1 = \frac{2 t^s}{\G(1+s)} \underset{\sigma\to 0^+}{\lim} ||\As P_\sigma f_k||_1
\\
& = \frac{2 t^s}{\G(1+s)} \underset{\sigma\to 0^+}{\lim} ||P_\sigma \As  f_k||_1
 = \frac{2 t^s}{\G(1+s)} ||\As f_k||_1,
\end{align*}
where in the third to the last inequality we have used Lemma \ref{L:pdec} with $p=1$, and the second to the last, and the last equalities follow from the fact that $\As P_\sigma f_k = P_\sigma \As  f_k \to \As  f_k$ in $L^1(\RN)$ as $\sigma \to 0^+$ (see \eqref{commutano} in Lemma \ref{L:comm} and Corollary \ref{C:Ptpzero}). Taking the $\underset{k\to \infty}{\liminf}$ in the latter inequality we conclude 
\[
||P_t \mathbf 1_E - \mathbf 1_E||_1 \le \frac{2 t^s}{\G(1+s)} {\bf{V}}^\sA_s(\mathbf 1_E;\{f_k\}).
\]
Passing to the infimum on all sequences $\{f_k\}_{k\in \mathbb N}\in \mathscr F(\mathbf 1_E)$, we finally obtain the fundamental inequality
\begin{equation}\label{perup}
||P_t \mathbf 1_E - \mathbf 1_E||_1 \le \frac{2 t^s}{\G(1+s)} \Ps(E).
\end{equation}
On the other hand, if \eqref{trace} is in force, Lemma \ref{L:perbelow} implies for some $c_N>0$,
\begin{equation}\label{perdown}
||P_t \mathbf 1_E - \mathbf 1_E||_1 \ge |E| -  \frac{c_N}{V(t/2)}  |E|^2.
\end{equation}
Combining \eqref{perup} with \eqref{perdown}, and using the hypothesis \eqref{vol0},  we conclude that the following basic interpolation estimate holds for every $t>0$,
\begin{equation}\label{perboth}
|E| \le \frac{2 \Ps(E)}{\G(1+s)}\ t^s + c_N\gamma_D^{-1} 2^{D/2} |E|^2\ t^{-D/2}.
\end{equation}
Minimising the function in the right-hand side of \eqref{perboth} with respect to $t>0$, we easily infer the desired conclusion \eqref{isoper20}
for some constant $i(s)>0$ depending exclusively on $N,D, \gamma_D$ and $s$.

\end{proof}

\begin{remark}\label{R:class}
Concerning \eqref{perup} we note that it trivially implies
\begin{equation}\label{sups}
\underset{t\to 0^+}{\limsup}\ \frac{1}{t^s} ||P_t \mathbf 1_E - \mathbf 1_E||_1 \le \underset{t>0}{\sup}\ \frac{1}{t^s} ||P_t \mathbf 1_E - \mathbf 1_E||_1 \le \frac{2}{\G(1+s)} \Ps(E).
\end{equation}
In this connection, we mention that, when $\sA = \Delta$ in \eqref{A0}, and therefore $P_t$ is the classical heat semigroup in $\RN$, it was proved in \cite[Theorem 3.3]{MPPP} that if $E\subset \RN$ is a measurable set having finite De Giorgi's perimeter $P(E)$, then 
\[
\underset{t\to 0^+}{\lim} \sqrt{\frac{\pi}{t}}\ ||P_t \mathbf 1_E - \mathbf 1_E||_1 = 2 P(E).
\]
\end{remark}

As we have mentioned in the introduction, the assumption \eqref{vol0} forces the condition $D_0 \le D_\infty$. Therefore, Theorem \ref{T:iso01} does not cover situations in which $D_0>D_\infty$. We next prove the second main result in this paper which covers such case.

\begin{proof}[Proof of Theorem \ref{T:isoother0}]

We can always assume $\Ps(E) <\infty$. Following the arguments in the proof of Theorem \ref{T:iso01}, we again obtain \eqref{perup} and \eqref{perdown}. At this point, we use the hypothesis \eqref{vol20}, instead of \eqref{vol0},  obtaining
$$
|E|-\frac{c_N |E|^2}{\gamma \min\left\{\left(\frac{t}{2}\right)^{\frac{D_0}{2}},\left(\frac{t}{2}\right)^{\frac{D_\infty}{2}}\right\}}\leq  \frac{2 t^s}{\G(1+s)} \Ps(E).
$$
This gives for every $t>0$,
\begin{equation}\label{przm}
|E| \le H(t) \overset{def}{=} \frac{2 \Ps(E)}{\G(1+s)}\ t^s + c_N\gamma^{-1}2^{\frac{D_0}{2}} |E|^2\ \max\left\{t^{-\frac{D_0}{2}},t^{\frac{-D_\infty}{2}}\right\}.
\end{equation}
If $\Ps(E)=0$, the previous estimate implies $|E|=0$, and we are done. We can thus assume $\Ps(E)>0$. Denote $c=\frac{\G(1+s)}{4s}\frac{c_N}{\gamma}2^{\frac{D_0}{2}}$, and consider the quantities
$$A_0=c D_0 \frac{|E|^2}{\Ps(E)}, \quad A_\infty=c D_\infty \frac{|E|^2}{\Ps(E)}.$$
The assumption $D_0>D_\infty$ implies that $A_0>A_\infty$. We now distinguish three cases: 
\begin{itemize}
\item[(i)] $A_\infty < A_0 \le 1$;
\item[(ii)] $1\le A_\infty < A_0$;
\item[(iii)] $A_\infty < 1 \le A_0$.
\end{itemize}
In the case (i), we have $A_0^{-\frac{D_\infty}{D_0+2s}} \le A_0^{-\frac{D_0}{D_0+2s}}$. Since \eqref{przm} holds for any $t>0$, we have in particular,
$$|E| \le H\left(A_0^{\frac{2}{D_0+2s}}\right)= \frac{2 \Ps(E)}{\G(1+s)}\ A_0^{\frac{2s}{D_0+2s}} + c_N\gamma^{-1}2^{\frac{D_0}{2}} |E|^2A_0^{-\frac{D_0}{D_0+2s}}.$$
From the previous inequality, after substituting the definition of $A_0$, we obtain
\begin{equation}\label{caso1iso}
|E|^{\frac{D_0-2s}{D_0}}\leq c_0 \Ps(E),
\end{equation}
for some positive constant $c_0=c_0(s, N, D_0, \gamma)$. If instead (ii) holds, we have $A_\infty^{-\frac{D_0}{D_\infty+2s}} \le A_\infty^{-\frac{D_\infty}{D_\infty+2s}}$, and we find
$$|E| \le H\left(A_\infty^{\frac{2}{D_\infty+2s}}\right)= \frac{2 \Ps(E)}{\G(1+s)}\ A_\infty^{\frac{2s}{D_\infty+2s}} + c_N\gamma^{-1}2^{\frac{D_0}{2}} |E|^2A_\infty^{-\frac{D_\infty}{D_\infty+2s}}.$$
Recalling the definition of $A_\infty$, from this inequality we deduce
\begin{equation}\label{caso2iso}
|E|^{\frac{D_\infty-2s}{D_\infty}}\leq c_\infty \Ps(E),
\end{equation}
for some positive constant $c_\infty=c_\infty(s, N, D_\infty, D_0,\gamma)$. Finally, if (iii) holds, then we have
$$1= c\left(D_\infty + \alpha(D_0-D_\infty)\right) \frac{|E|^2}{\Ps(E)}\quad\mbox{ for some }\alpha\in(0,1].$$
If we use the inequality 
$$
|E| \le H(1) = H\left(\left( c\left(D_\infty + \alpha(D_0-D_\infty)\right) \frac{|E|^2}{\Ps(E)} \right)^{\frac{2}{D_0+2s}}\right)$$
we obtain
\begin{align*}
|E| &\le \frac{2 \Ps(E)}{\G(1+s)}\ \left( c\left(D_\infty + \alpha(D_0-D_\infty)\right) \frac{|E|^2}{\Ps(E)} \right)^{\frac{2s}{D_0+2s}}\\ &+ c_N\gamma^{-1}2^{\frac{D_0}{2}} |E|^2\ \left( c\left(D_\infty + \alpha(D_0-D_\infty)\right) \frac{|E|^2}{\Ps(E)} \right)^{-\frac{D_0}{D_0+2s}}.
\end{align*}
If instead we use
$$
|E| \le H(1) = H\left(\left( c\left(D_\infty + \alpha(D_0-D_\infty)\right) \frac{|E|^2}{\Ps(E)} \right)^{\frac{2}{D_\infty+2s}}\right)$$
we find
\begin{align*}
|E| &\le \frac{2 \Ps(E)}{\G(1+s)}\ \left( c\left(D_\infty + \alpha(D_0-D_\infty)\right) \frac{|E|^2}{\Ps(E)} \right)^{\frac{2s}{D_\infty+2s}}\\ &+ c_N\gamma^{-1}2^{\frac{D_0}{2}} |E|^2\ \left( c\left(D_\infty + \alpha(D_0-D_\infty)\right) \frac{|E|^2}{\Ps(E)} \right)^{-\frac{D_\infty}{D_\infty+2s}}.
\end{align*}
In both cases, since $\alpha\in (0,1]$, after some elementary computations we obtain the existence of some positive $c_1=c_1(s, N, D_\infty, D_0, \gamma)$ such that
\begin{equation}\label{caso3iso}
|E|^{\frac{D_0-2s}{D_0}},\,|E|^{\frac{D_\infty-2s}{D_\infty}}\leq c_1 \Ps(E).
\end{equation}
Putting together \eqref{caso1iso}, \eqref{caso2iso} and \eqref{caso3iso}, we finally infer \eqref{isoperother20}.

\end{proof}

The following isoperimetric inequality for $\mathfrak P^{\sA,\star}_s$ is an immediate consequence of Theorem \ref{T:iso01} and Lemma \ref{L:senzastarminoredistar}.

\begin{corollary}\label{C:iso}
Suppose that $D_0 \le D_\infty$, and that \eqref{vol0} hold. For every $0<s< \frac{1}{2}$ there exists a constant $i(s) >0$, depending on $N,D,s,\gamma_D$, such that for any bounded measurable set $E\subset \RN$,  one has
$$\mathfrak P^{\sA,\star}_s(E)\ \ge\ i(s)\ |E|^{(D-2s)/D}.$$
\end{corollary}

We now specialise Corollary \ref{C:iso} to the situation in which $\sA = \Delta$. In such case, we have $V(t) = \omega_N t^{N/2}$ and therefore $D_0 = D_\infty = N\ge 2$. 

\begin{corollary}\label{C:isoE}
Suppose that $\sA = \Delta$ in \eqref{A0}. For every $0<s<\frac{1}{2}$ there exists a constant $i(N,s)>0$ such that for any bounded measurable set $E\subset \RN$,  one has
$$\mathfrak P^{\Delta,\star}_s(E)\ \ge\ i(N,s)\ |E|^{(N-2s)/N}.$$
\end{corollary}

Combining Proposition \ref{P:equalstars} with Corollary \ref{C:isoE} we obtain the isoperimetric inequality \eqref{isos} for the nonlocal perimeter $P_s$. Such result is not new since, as we have mentioned in the introduction,  it can be extracted from the work \cite{AL}. Our proof, based on the heat semigroup, provides a different perspective on it. 

The next result provides an interesting one-sided bound for the limiting case $s=1/2$ similar to the right-hand side of the  Bourgain, Brezis and Mironescu's bound \eqref{anotherlook}.

\begin{proposition}\label{P:star}
Suppose that \eqref{trace} hold. Then, for every bounded measurable set $E\subset \RN$ one has
\begin{equation}\label{ve2}
\underset{s\nearrow 1/2}{\limsup}\ (1/2 - s)\ \mathfrak P^{\sA,\star}_s(E)  \le  \underset{\tau >0}{\sup}\ \frac{1}{\sqrt{4\pi \tau}} ||P_\tau\In - \In||_1.
\end{equation}  
\end{proposition}

\begin{proof}
If in \eqref{ve2} we have $\underset{\tau >0}{\sup}\ \frac{1}{\sqrt{\tau}} ||P_\tau\In - \In||_1 = \infty$, there is nothing to prove, thus we might as well assume that such quantity be finite. For every $0<s<1/2$ we thus have 
\[
\int_0^1 \frac{1}{\tau^{1+s}} ||P_\tau\In - \In||_1 \le \underset{\tau >0}{\sup}\ \frac{1}{\sqrt{\tau}} ||P_\tau\In - \In||_1 \int_0^1 \frac{d\tau}{\tau^{1+s-1/2}} < \infty.
\] 
Since on the other hand by (iv) of Lemma \ref{L:Pt} we have
\[
\int_1^\infty \frac{1}{\tau^{1+s}} ||P_\tau\In - \In||_1 d\tau \le 2 |E| \int_1^\infty \frac{d\tau}{\tau^{1+s}} < \infty,
\] 
by Lemma \ref{L:comm} we infer that $\In, P_t \In\in D_1(\As)$ for all $t>0$, and that \eqref{commutano} holds. But then, we have
\begin{equation}\label{PstarAs}
\mathfrak P^{\sA,\star}_s(E) = \underset{t\to 0^+}{\lim}\ ||\As P_t  \In||_1 = \underset{t\to 0^+}{\lim}\ ||P_t \As \In||_1 = ||\As \In||_1,
\end{equation} 
where in the last equality we have used Corollary \ref{C:Ptpzero}.
With this being said, for any $\ve>0$ we obtain from Lemma \ref{L:alteAs},
\begin{align*}
||(-\sA)^s \In||_1 & =  \frac{s}{\G(1-s)} \int_0^\ve \frac{1}{\tau^{1+s}} ||P_\tau\In - \In||_1 d\tau
+ \frac{s}{\G(1-s)} \int_\ve^\infty \frac{1}{\tau^{1+s}} ||P_\tau\In - \In||_1 d\tau.
\end{align*}
One easily recognises
\[
\int_\ve^\infty \frac{1}{\tau^{1+s}} ||P_\tau\In - \In||_1 d\tau \le \frac{2 |E|}{s}\ \ve^{-s}.
\] 
On the other hand, one has
\[
\int_0^\ve \frac{1}{\tau^{1+s}} ||P_\tau\In - \In||_1 d\tau \le \underset{\tau>0}{\sup}\ \frac{1}{\sqrt \tau} ||P_\tau\In - \In||_1\ \frac{\ve^{1/2 - s}}{1/2 - s}.
\]
We infer that for every $\ve>0$ we have
\begin{equation}\label{ve}
||(-\sA)^s \In||_1 \le \frac{s}{\G(1-s)\left(1/2 - s\right)}\ \underset{\tau>0}{\sup}\ \frac{1}{\sqrt \tau} ||P_\tau\In - \In||_1\ \ve^{1/2 - s} + \frac{2 |E|}{\G(1-s)}\ \ve^{-s}.
\end{equation}
Minimising with respect to $\ve>0$ the right-hand side of \eqref{ve}, and using \eqref{PstarAs}, we find
\begin{equation}\label{ve22}
\mathfrak P^{\sA,\star}_s(E) \le 2^{1-2s} \frac{s}{\G(1-s)}\ |E|^{1-2s} \left(\underset{\tau>0}{\sup}\ \frac{1}{\sqrt \tau} ||P_\tau\In - \In||_1\right)^{2s} \left\{ \frac{1}{1/2 - s} + \frac 1s\right\}.
\end{equation}
Multiplying \eqref{ve22} by $1/2 - s$, and taking the $\limsup$, we easily reach the conclusion 
\[
\underset{s\nearrow 1/2}{\limsup}\ (1/2 - s)\ \mathfrak P^{\sA,\star}_s(E)  \le  \underset{\tau >0}{\sup}\ \frac{1}{\sqrt{4\pi \tau}} ||P_\tau\In - \In||_1,
\]
which proves \eqref{ve2}. 
\end{proof}

\section{Applications: a strong geometric embedding}\label{S:app}

As we have said in the introduction, the aim of this section is to present a notable application of the isoperimetric inequality \eqref{isoper20} in Theorem \ref{T:iso01}. We establish a nonlocal version of the beautiful end-point equivalence, for the classical Sobolev spaces,  of the weak embedding $W^{1,1}(\RN) \hookrightarrow L^{N/(N-1),\infty}$, with the strong one $W^{1,1}(\RN) \hookrightarrow L^{N/(N-1)}$, and with the isoperimetric inequality $P(E) \ge C_N |E|^{\frac{N-1}N}$. Our final objective is proving Theorem \ref{T:strongsob0} and Theorem \ref{T:strongsob2} for an optimal class of Besov spaces, introduced in \cite[Section 3]{GTfi}, and naturally associated with the operator $\sA$ in \eqref{A0}. We begin by recalling the relevant definition.
 
\begin{definition}\label{D:besov}
For $p\geq 1$ and $\alpha\geq 0$, we define the \emph{Besov space} $\Bpa$ as the collection of those functions $f\in L^p$, such that the seminorm
\begin{equation}\label{besov}
\mathscr N_{\alpha,p}(f) = \left(\int_0^\infty  \frac{1}{t^{1+\frac{\alpha p}2}} \int_{\RN} P_t\left(|f - f(X)|^p\right)(X) dX dt\right)^{\frac 1p} < \infty.
\end{equation}
We endow the space $\Bpa$ with the following norm
$$||f||_{\Bpa} \overset{def}{=} ||f||_{p} + \mathscr N_{\alpha,p}(f).$$
\end{definition}

\begin{remark}
Before proceeding, we pause to remark that the spaces $\Bpa$ represent a generalisation of the classical Aronszajn-Gagliardo-Slobedetzky spaces $W^{\alpha,p}(\RN)$, defined via the seminorm
\[
\int_{\RN} \int_{\RN} \frac{|f(X) - f(Y)|^p}{|X-Y|^{N+\alpha p}} dX dY.
\]
 This is easily recognised from \eqref{besov} since, when $Q = I_N$ and $B = O_N$, and thus $\sA = \Delta$, one has $p(X,Y,t) = (4\pi t)^{-N/2} e^{-\frac{|X-Y|^2}{4t}}$.
\end{remark}

Having observed this, we next emphasise that the spaces are non-trivial. One has in fact the following.
\begin{lemma}
Assume \eqref{trace}. For any $p\geq 1$ and $0<\alpha<1$ we have $\So\subseteq\Bpa$.
\end{lemma}

\begin{proof}
Let $\tau_0\in (0,1]$ and $M>0$ be such that for all $X\in\RN$ and for all $\tau\in (0,\tau_0]$ one has
\begin{equation}\label{choicetau0}
|(e^{\tau B}-\mathbb{I}_N)X)|\leq M \tau |X|\leq \frac{1}{2}|X|\ \  \text{and}\ \ |K^{-\frac{1}{2}}(\tau) X|\geq M^{-1}\ |X|.
\end{equation}
For $f\in \So$ and $t \in (0,\tau_0)$ we now have
\begin{align*}
& \int_{\RN} P_t(|f - f(X)|^p)(X) dX = \int_{\RN}\int_{\RN} p(X,Y,t)|f(Y) - f(X)|^p dYdX
\\
& = \int_{\RN}\int_{\RN} p(0,Z,t)|f(Z+e^{tB}X) - f(X)|^p dZdX  
\\
& = \int_{\{|Z|> 1\}\cup \{|Z|\leq 1\}}p(0,Z,t) \int_{\RN} |f(Z+e^{tB}X) - f(X)|^p dX dZ.
\end{align*}
On the one hand, we have for some $c_p>0$,
\begin{align*}
&\int_{\{|Z|> 1\}}p(0,Z,t)\left(\int_{\RN}\left|f(Z+e^{tB}X) - f(X)\right|^p dX\right)dZ \\
&\leq c_p \int_{\{|Z|> 1\}}p(0,Z,t)\left(\int_{\RN}\left|f(Z+e^{tB}X)\right|^p + \left|f(X)\right|^p dX\right)dZ\\
&= c_p(e^{-t\operatorname{tr} B } + 1)||f||^p_p\int_{\{|Z|> 1\}}p(0,Z,t) dZ\leq 2c_p ||f||^p_p\int_{\{|Z|> 1\}}p(0,Z,t)|Z|^p dZ \\
&\leq 2 M^p c_p ||f||^p_p t^{\frac{p}{2}} \int_{\RN}p(0,Z,t)\frac{\left|K^{-\frac{1}{2}}(\tau) Z\right|^p}{t^{\frac{p}{2}}} dZ = 2 M^p c_p c_N \left(\int_{\RN}|X|^pe^{-\frac{|X|^2}{4}} dX\right)||f||^p_p t^{\frac{p}{2}},
\end{align*}
where we have used \eqref{choicetau0} and the expression \eqref{p}. On the other hand, using \eqref{choicetau0} again, we obtain
\begin{align*}
&\int_{\{|Z|\leq 1\}}p(0,Z,t)\left(\int_{\RN}\left|f(Z+e^{tB}X) - f(X)\right|^p dX\right)dZ \\
&\leq \int_{\{|Z|\leq 1\}}p(0,Z,t)\left(\int_{\RN} \left|Z+\left(e^{tB}-\mathbb{I}_N\right)X\right|^p \sup_{B(X, 1+\frac{1}{2}|X|)}\left|\nabla f\right|^p dX\right)dZ\\
&= c_p \int_{\{|Z|\leq 1\}}p(0,Z,t)\left(\int_{\RN} \left(|Z|^p + M^p t^p |X|^p\right) \sup_{B(X, 1+\frac{1}{2}|X|)}\left|\nabla f\right|^p dX\right)dZ \\
&\leq c_pM^p t^{\frac{p}{2}} \left(\int_{\RN}p(0,Z,t)\frac{\left|K^{-\frac{1}{2}}(\tau) Z\right|^p}{t^{\frac{p}{2}}} dZ\right)\left( \int_{\RN}\sup_{B(X, 1+\frac{1}{2}|X|)}\left|\nabla f\right|^p dX\right)\\ 
&+ c_pM^pt^{p} \left(\int_{\RN}|X|^p\sup_{B(X, 1+\frac{1}{2}|X|)}\left|\nabla f\right|^p dX\right).
\end{align*}
All the integrals involving the term $|\nabla f|$ are finite since $f\in\So$. Hence we have just showed that
$$
\int_{\RN} P_t\left(|f - f(X)|^p\right)(X) dX \leq C_f \,\cdot \,t^{\frac{p}{2}} \quad \mbox{for all }t\in (0,\tau_0)
$$
for some positive constant $C_f$. Therefore, since $\alpha <1$, we have
$$
\int_0^{\tau_0}  \frac{1}{t^{1+\frac{\alpha p}2}} \int_{\RN} P_t\left(|f - f(X)|^p\right)(X) dX dt <\infty
$$
whenever $f\in\So$. Since $\alpha>0$ and $\operatorname{tr} B\geq 0$, one can easily see that also in the interval $(\tau_0,\infty)$ the integral under discussion can be bounded above in terms of $||f||_p$. This proves that $\mathscr N_{\alpha,p}(f)<\infty$ for all $f\in\So$.

\end{proof}

We next establish a nonlocal coarea formula involving the Besov-seminorm $\mathscr N_{2s,1}(f)$.

\begin{proposition}[Coarea formula]\label{coarea}
Let $s\in (0,\frac{1}{2})$. For any $f\in \Bs$ we have
\begin{equation}\label{chiaveinstrong}
\mathscr N_{2s,1}(f)= \frac{\G(1-s)}{s} \int_\R || \As \mathbf 1_{\{f>\sigma\}} ||_{1} d\sigma \geq \frac{\G(1-s)}{s} \int_\R \mathfrak P^{\sA}_s\left(\{f>\sigma\}\right) d\sigma.
\end{equation}
In particular, for any nonnegative $f\in\So$, we have
$$\mathscr N_{2s,1}(f)=\frac{\G(1-s)}{s} \int_0^{+\infty} \mathfrak P^{\sA,\star}_s\left(\{f>\sigma\}\right) d\sigma.$$
\end{proposition}
\begin{proof}
We first prove that, for any measurable function $f$, we have
\begin{equation}\label{introbes}
\mathscr N_{2s,1}(f)= \int_\R \int_0^\infty \frac{1}{t^{1+s}}|| P_t \mathbf 1_{\{f>\sigma\}} - \mathbf 1_{\{f>\sigma\}} ||_{1} dt d\sigma.
\end{equation}
To see this, for any $\sigma\in\R$, we denote $E_\sigma=\{X\in\RN\,:\, f(X)>\sigma\}$ and $E^c_\sigma=\RN\smallsetminus E_\sigma$. Since $P_t 1=1$ for all $t>0$, we have
\begin{align*}
|| P_t \mathbf 1_{E_{\sigma}} - \mathbf 1_{E_{\sigma}} ||_{1} &=\int_{E^c_\sigma} P_t \mathbf 1_{E_{\sigma}}(X) dX + \int_{E_{\sigma}} \left(1 - P_t \mathbf 1_{E_{\sigma}}(X)\right) dX \\
& =\int_{\RN} \mathbf 1_{E^c_\sigma}(X) P_t \mathbf 1_{E_{\sigma}}(X) dX + \int_{\RN}  \mathbf 1_{E_{\sigma}}(X) P_t \left(1-\mathbf 1_{E_{\sigma}}\right)(X) dX\\
&=\int_{\RN}\int_{\RN} p(X,Y,t)\left(\mathbf 1_{E^c_\sigma}(X) \mathbf 1_{E_{\sigma}}(Y) + \mathbf 1_{E_{\sigma}}(X)\mathbf 1_{E^c_\sigma}(Y)\right) dYdX\\
&=\int_{\RN}\int_{\RN} p(X,Y,t)\left|\mathbf 1_{E_{\sigma}}(Y) - \mathbf 1_{E_{\sigma}}(X)\right| dYdX.
\end{align*}
Having in mind that $\int_{\R} \left|\mathbf 1_{E_{\sigma}}(Y) - \mathbf 1_{E_{\sigma}}(X)\right|d\sigma = |f(Y)-f(X)|$, by Tonelli's theorem we obtain
\begin{align*}
\int_\R \int_0^\infty \frac{1}{t^{1+s}}|| P_t \mathbf 1_{E_{\sigma}} - \mathbf 1_{E_{\sigma}} ||_{1} dt d\sigma &= \int_0^\infty \frac{1}{t^{1+s}} \int_{\RN}\int_{\RN} p(X,Y,t)|f(Y)-f(X)| dYdXdt \\
&= \int_0^\infty \frac{1}{t^{1+s}} \int_{\RN}P_t\left(|f - f(X)|\right)(X)dXdt,
\end{align*}
which proves \eqref{introbes}. In particular, if $f\in \Bs$, we know from \eqref{introbes} that $\int_0^\infty \frac{1}{t^{1+s}}|| P_t \mathbf 1_{E_{\sigma}} - \mathbf 1_{E_{\sigma}} ||_{1} dt<\infty$ for almost every $\sigma$. Then, for those values of $\sigma$, by Lemma \ref{L:comm} and Lemma \ref{L:alteAs} (see also \cite[Proposition 3.4 and Remark 3.5]{GTfi}) we infer that $\mathbf 1_{E_{\sigma}}\in D_{1,s}$, and
$$
\mathscr N_{2s,1}(f)= \frac{\G(1-s)}{s} \int_\R || \As \mathbf 1_{E_{\sigma}} ||_{1} d\sigma.
$$
By Lemma \ref{L:perAs} we know that 
$$
\mathfrak P^{\sA}_s\left(E_{\sigma}\right)\leq || \As \mathbf 1_{E_{\sigma}} ||_{1}.
$$
This completes the proof of \eqref{chiaveinstrong}. To prove the last statement, we have to keep in mind Definition \ref{per2} and the fact that $E_\sigma$ is bounded for all $\sigma>0$ if $f\in\So$ is nonnegative. Then, from Corollary \ref{corinfondo} we know that $|| \As \mathbf 1_{E_{\sigma}} ||_{1}=\mathfrak P^{\sA,\star}_s\left(E_{\sigma}\right)$, and we conclude that
$$\mathscr N_{2s,1}(f)= \frac{\G(1-s)}{s} \int_\R || \As \mathbf 1_{\{f>\sigma\}} ||_{1} d\sigma= \frac{\G(1-s)}{s}\int_\R \mathfrak P^{\sA,\star}_s\left(\{f>\sigma\}\right) d\sigma.$$
\end{proof}
In \cite[Proposition 3.3]{GTfi} we proved the boundedness of the map $\As: \Bs\longrightarrow L^1$. This says that $\Bs\hookrightarrow D_{1,s}$. On the other hand, thanks to Proposition \ref{density} and \cite[Theorem 7.5]{GThls}, we know that, under the assumption \eqref{vol0}, we have
$D_{1,s}=\mathscr{L}^{2s,1} \hookrightarrow L^{\frac{D}{D-2s},\infty}$. If we combine these facts, we obtain 
\begin{equation}\label{weak}
\Bs  \hookrightarrow L^{\frac{D}{D-2s},\infty}.
\end{equation}
The final objective of this section is to show that, in \eqref{weak}, we can replace $L^{\frac{D}{D-2s},\infty}$ with the strong space $L^{\frac{D}{D-2s}}$. 
We will need the following real analysis lemma, whose proof is left to the reader.
\begin{lemma}\label{L:ra}
Let $G:[0,\infty)\to [0,\infty)$ be a non-increasing function. Then, for any $D>2s>0$ we have
\[
\frac{D}{D-2s} \int_0^\infty t^{\frac{2s}{D-2s}} G(t) dt \le \left(\int_0^\infty G(t)^{\frac{D-2s}D} dt\right)^{\frac{D}{D-2s}}.
\]
\end{lemma}

We are ready to prove the main result of this section.

\begin{proof}[Proof of Theorem \ref{T:strongsob0}]
Let $f\in\Bs$. For $\sigma\geq 0$, we consider $E_\sigma=\{X\in\RN\mid |f(X)|>\sigma\}$, and we define $G(\sigma)=|E_\sigma|$. Since $G$ is non-increasing and $D\geq D_0\ge 2 >2s$, Lemma \ref{L:ra} gives
\[
||f||_{\frac{D}{D-2s}}=\left(\int_{\RN} |f|^\frac{D}{D-2s}(X) dX\right)^\frac{D-2s}{D} = \left(\frac{D}{D-2s} \int_0^\infty \sigma^{\frac{2s}{D-2s}} G(\sigma) d\sigma\right)^\frac{D-2s}{D}\leq \int_0^\infty G(\sigma)^{\frac{D-2s}D} d\sigma.
\]
We can now apply Theorem \ref{T:iso01}, which implies
\[
\int_0^\infty G(\sigma)^{\frac{D-2s}D} d\sigma=\int_0^\infty |E_\sigma|^{\frac{D-2s}D} d\sigma \le \frac{1}{i(s)} \int_0^\infty \mathfrak P^{\sA}_s\left(E_\sigma\right) d\sigma.
\]
At this point we observe that \eqref{chiaveinstrong} gives
\[
\int_0^\infty \mathfrak P^{\sA}_s\left(E_\sigma\right) d\sigma \leq \frac{s}{\G(1-s)} \mathscr N_{2s,1}(|f|)\leq \frac{s}{\G(1-s)} \mathscr N_{2s,1}(f).
\]
Combining the latter three inequalities, we obtain the desired conclusion \eqref{ssob}.

\end{proof}

Analogously to the two different isoperimetric inequalities established with Theorem \ref{T:iso01} and \ref{T:isoother0}, we have a substitute result for Theorem \ref{T:strongsob0} in case $D_0>D_\infty$.

\begin{theorem}\label{T:strongsob2}
Let $s\in(0,\frac{1}{2})$ and $D_0>D_\infty$. Suppose that \eqref{trace} be valid, and that there exists $\gamma>0$ such that \eqref{vol20} hold. Then we have
$$\Bs\hookrightarrow L^{\frac{D_0}{D_0-2s}}+L^{\frac{D_\infty}{D_\infty-2s}}.$$ 
Precisely, for every $f\in \Bs$ one has
$$||f||_{L^{\frac{D_0}{D_0-2s}}+L^{\frac{D_\infty}{D_\infty-2s}}} \le \frac{2s}{i(s)\G(1-s)} \mathscr N_{2s,1}(f),$$
where $i(s)>0$ is the constant appearing in Theorem \ref{T:isoother0}.
\end{theorem}
\begin{proof}
Set $q_0=\frac{D_0}{D_0-2s}$ and $q_\infty=\frac{D_\infty}{D_\infty-2s}$. Under our assumptions we know that $D_0>D_\infty\geq 2>2s$. We also recall that, when we write $L^{q_0}+L^{q_\infty}$, we mean the Banach space of functions $f$ which can be written as $f=f_1+f_2$ with $f_1\in L^{q_0}$ and $f_2\in L^{q_\infty}$, endowed with the norm
$$||f||_{L^{q_0}+L^{q_\infty}} = \inf_{f=f_1+f_2,\\ f_1\in L^{q_0},\,f_2\in L^{q_\infty}}{||f_1||_{q_0}+||f_2||_{q_\infty}}.$$
Fix $f\in\Bs$, and assume for simplicity $f\geq 0$. For $\sigma\geq 0$, we consider $E_\sigma=\{X\in\RN\mid\ f(X)>\sigma\}$. Denote $$\sigma_f=\sup\{\sigma> 0\,:\, |E_\sigma|> 1\}.$$ If $|E_\sigma|\leq 1$ for all $\sigma$, we agree to let $\sigma_f=0$. We note that $\sigma_f\in [0,\infty)$ since $f\in L^1$. We denote
$$f_1(X)=f(X)\chi_{E_{\sigma_f}}(X)\quad\mbox{and}\quad f_2(X)=f(X)-f_1(X).$$
We make use of the notation $E^i_\sigma=\{X\in\RN\,:\,f_i(X)>\sigma\}$ for $i\in \{1,2\}$. The following holds:
$$E^1_\sigma=\begin{cases}
E_\sigma \qquad\,\,\mbox{if }\sigma>\sigma_f,  \\
E_{\sigma_f} \qquad \mbox{if }\sigma\leq\sigma_f
\end{cases}\quad\mbox{ and }\quad E^2_\sigma=\begin{cases}
\varnothing \,\,\,\quad\quad\quad\qquad\mbox{if }\sigma>\sigma_f,  \\
E_\sigma\smallsetminus E_{\sigma_f} \qquad \mbox{if }\sigma\leq\sigma_f.
\end{cases}$$
From Lemma \ref{L:ra} we obtain
\[
||f_1||_{q_0}=\left(\int_{\RN} f_1^\frac{D_0}{D_0-2s}(X) dX\right)^\frac{D_0-2s}{D_0} = \left(\frac{D_0}{D_0-2s} \int_0^\infty \sigma^{\frac{2s}{D_0-2s}} |E^1_\sigma| d\sigma\right)^\frac{D_0-2s}{D_0}\leq \int_0^\infty |E^1_\sigma|^{\frac{D_0-2s}{D_0}} d\sigma.
\]
We now notice that $|E^1_\sigma|\leq 1$ for all $\sigma$. In fact, Beppo Levi's monotone convergence theorem ensures that $|E_{\sigma_f}|=\sup_{\sigma>\sigma_f}|E_{\sigma}|\leq 1$. Hence, we have
\begin{align}\label{funo}
||f_1||_{q_0}&\leq \int_0^\infty |E^1_\sigma|^{\frac{D_0-2s}{D_0}} d\sigma = \int_0^\infty \min\left\{|E^1_\sigma|^{\frac{D_0-2s}{D_0}},|E^1_\sigma|^{\frac{D_\infty-2s}{D_\infty}}\right\} d\sigma\nonumber\\
&\leq \int_0^\infty \min\left\{|E_\sigma|^{\frac{D_0-2s}{D_0}},|E_\sigma|^{\frac{D_\infty-2s}{D_\infty}}\right\} d\sigma.
\end{align}
On the other hand, from Lemma \ref{L:ra} we also have
\begin{align*}
||f_2||_{q_\infty}&=\left(\int_{\RN} f_2^\frac{D_\infty}{D_\infty-2s}(X) dX\right)^\frac{D_\infty-2s}{D_\infty} = \left(\frac{D_\infty}{D_\infty-2s} \int_0^\infty \sigma^{\frac{2s}{D_\infty-2s}} |E^2_\sigma| d\sigma\right)^\frac{D_\infty-2s}{D_\infty}\leq \int_0^\infty |E^2_\sigma|^{\frac{D_\infty-2s}{D_\infty}} d\sigma\\
&= \int_0^{\sigma_f} |E^2_\sigma|^{\frac{D_\infty-2s}{D_\infty}} d\sigma\leq \int_0^{\sigma_f} |E_\sigma|^{\frac{D_\infty-2s}{D_\infty}} d\sigma.
\end{align*}
Since in the interval $(0,\sigma_f)$ we know that $|E_\sigma|>1$, we deduce
\begin{equation}\label{fdue}
||f_2||_{q_\infty}\leq \int_0^{\sigma_f} |E_\sigma|^{\frac{D_\infty-2s}{D_\infty}} d\sigma=\int_0^{\sigma_f} \min\left\{|E_\sigma|^{\frac{D_0-2s}{D_0}},|E_\sigma|^{\frac{D_\infty-2s}{D_\infty}}\right\} d\sigma.
\end{equation}
Combining \eqref{funo} and \eqref{fdue} we infer
$$||f||_{L^{q_0}+L^{q_\infty}}\leq ||f_1||_{q_0}+||f_2||_{q_\infty}\leq 2\int_0^\infty \min\left\{|E_\sigma|^{\frac{D_0-2s}{D_0}},|E_\sigma|^{\frac{D_\infty-2s}{D_\infty}}\right\} d\sigma.$$
We can now exploit the isoperimetric inequality in Theorem \ref{T:isoother0} and the coarea formula in Proposition \ref{coarea} to conclude that
$$||f||_{L^{q_0}+L^{q_\infty}}\leq \frac{2}{i(s)}\int_0^\infty \mathfrak P^{\sA}_s\left(E_\sigma\right) d\sigma\leq \frac{2s}{i(s)\G(1-s)} \mathscr N_{2s,1}(f).$$
\end{proof}




\bibliographystyle{amsplain}

\end{document}